\documentclass{amsart}

\usepackage[utf8]{inputenc}
\usepackage{amsmath,amsthm,amsfonts,amssymb}
\usepackage[utf8]{inputenc}
\usepackage{calc}
\usepackage{cases}
\usepackage{accents}
\usepackage{mathtools}
\usepackage{physics}
\usepackage{amsfonts}
\usepackage{amssymb}
\usepackage{graphicx}
\usepackage{enumitem}
\usepackage{mathrsfs}
\usepackage[all]{xy}
\usepackage{datetime}
\usepackage{cleveref}

\newlist{steps}{enumerate}{1}
\setlist[steps, 1]{label = Step \arabic*:}

\newcommand{\PP}[1]{\mathbb{P}_{#1}}

\newcommand{\gradb}[1]{\overline{\nabla}^{#1}}
\newcommand{\slam}{\sqrt{\lambda}}

\newcommand{\rtp}{\mathbb{R}^3_{+}}
\newcommand{\ttp}{\mathbb{R}_{+}\times\mathbb{T}^2}
\newcommand{\btp}{\mathbb{R}_{+}\times\Omega}

\newcommand{\tbtp}{\mathbb{R}_{+}\times\tilde{\Omega}}

\newcommand{\rt}{\mathbb{R}^2}
\newcommand{\ts}{\mathbb{T}^2}

\newcommand{\btpz}{\lbrace z=0\rbrace\times\Omega}

\newcommand{\hh}{\mathcal{H}}
\newcommand{\hhh}{\mathcal{H}^2}
\newcommand{\dc}{\mathcal{D}_0}
\newcommand{\ddd}{\mathcal{D}_{\lambda}}
\newcommand{\glmd}{\gamma_\lambda}
\newcommand{\gz}{\gamma_{0}}
\newcommand{\tomega}{\tilde{\Omega}}
\newcommand{\Lapl}{\Delta_{\lambda}}
\newcommand{\dslam}{\nabla_{\sqrt{\lambda}}}
\newcommand{\zlz}{\partial_z\lambda\partial_{z}}

\DeclareMathOperator{\divg}{div}

\DeclareMathOperator{\supp}{supp}

\usepackage{soul}
\usepackage{bigints}
\usepackage{breqn}

\theoremstyle{definition}
\newtheorem{definition}{Definition}[section]
\theoremstyle{remark}
\newtheorem*{remark}{Remark}
\theoremstyle{plain}
\newtheorem{theorem}{Theorem}[section]
\newtheorem{lemma}[theorem]{Lemma}
\newtheorem{proposition}[theorem]{Proposition}

\title[SINGULAR 3D QG SYSTEMS]{Global in Time Weak Solutions to Singular 3D Quasi-Geostrophic Systems}

\author{Yiran Hu}
\address{Department of Mathematics, University of Texas at Austin, Austin, TX 78712}
\email{yrhu@math.utexas.edu} 
\thanks{\textbf{Acknowledgment:}   The author would like to thank her advisor, Alexis Vasseur for suggesting this problem.}
\thanks{This work was partially funded by NSF grant: 1840314, 2219434, 1907981. }

\usepackage{amsopn}

\begin{document}

\begin{abstract}
Geophysicists have studied 3D Quasi-Geostrophic systems extensively. These systems describe stratified flows in the atmosphere on a large time scale and are widely used for forecasting atmospheric circulation. They couple an inviscid transport equation in $\mathbb{R}_{+}\times\Omega$ with an equation on the boundary satisfied by the trace, where $\Omega$ is either $2D$ torus or a bounded convex domain in $\rt$. In this paper, we show the existence of global in time weak solutions to a family of singular 3D quasi-geostrophic systems with Ekman pumping, where the background density profile degenerates at the boundary. The proof is based on the construction of approximated models which combine the Galerkin method at the boundary and regularization processes in the bulk of the domain. The main difficulty is handling the degeneration of the background density profile at the boundary.
\end{abstract}
\maketitle
\\


\section{Introduction}
The 3D quasi-geostrophic system is a widely used model and has been analyzed in the book \textit{Geophysical Fluid Dynamics} by Pedlosky \cite{pedlosky2013geophysical}; the system is derived from the Euler equation gravity effect and takes into account the rotation of the earth. The flow is stratified in the region due to the geostrophic and hydrostatic balance. The model describes the evolution of $\Psi$, the fluctuation of the pressure. The so-called vorticity potential is then defined as an elliptic operator:

\[\Lapl\Psi=\partial_z\lambda\partial_z\Psi +\partial^2_{x_{1} x_{1}} \Psi+\partial^2_{x_{2} x_{2}} \Psi,\qquad\text{with~} \lambda(z)=z^a,~~ a>1.\]

The background density parameter $\lambda$ verifies $\lambda(z)=-\frac{1}{\partial_z\bar{\rho}}$ where $\bar{\rho}(z)$ is a reference density profile in $z$ only. The most well studied case $\lambda\equiv 1$ corresponds to $\bar{\rho}(z)=R-z$. In this paper, we consider a family of degenerate cases where $\bar{\rho}(z)=R-cz^{1-a}$ for $a<1$. The case of negative $a$ corresponds to a small variation of the reference density near the surface of the earth, while positive $a$ corresponds to the large variation of this quantity. Consider a domain $\Omega\in \rt$. We denote $x=(x_1,x_2)\in \Omega$, and we use the following notation: \[\gradb{}=(0,\partial_{x_1},\partial_{x_2}).\]
The velocity field of the stratified flow is then given by
\[\gradb{\perp}\Psi=(0,-\partial_{x_2}\Psi,\partial_{x_1}\Psi),\quad \overline{\Delta}=(0, \partial^2_{x_1}, \partial^2_{x_2}).\]
Note that the potential vorticity can be rewritten as,
\begin{equation}
\label{dlamnotation}
\Lapl \Psi:= \divg \left(\nabla_\lambda\Psi\right) =\divg\left(\lambda\partial_{z}\Psi,\partial_{x_1}\Psi,\partial_{x_2}\Psi\right)=\partial_z\lambda\partial_z\Psi +\partial^2_{x_{1} x_{1}} \Psi+\partial^2_{x_{2} x_{2}} \Psi.
\end{equation}
We denote $\glmd \Psi$ the function of $x$ and $t$ only, corresponding to the Neumann condition
$$
\theta = \glmd  \Psi(t, x)=-\lim\limits_{z=0}\lambda\partial_z \Psi(t, z, x).
$$
Then the $3D$ Quasi-geostrophic equation reads,
\begin{equation}
\label{3DQG}
\tag{QG}
\left\{\begin{array}{ll}
{\left(\partial_{t}+\overline{\nabla}^{\perp} \Psi \cdot \overline{\nabla}\right)\left(\Lapl \Psi\right)=0,} & {t>0, \quad z>0,\quad x\in \Omega}, \\
{\left(\partial_{t}+\overline{\nabla}^{\perp} \Psi \cdot \overline{\nabla}\right) \left(\gamma_\lambda \Psi\right)=\overline{\Delta} \Psi,} & {t>0, \quad z=0,\quad x\in \Omega}, \\
{\Psi(0, z, x)=\Psi_{0}(z, x),} & {t=0, \quad z \geq 0,\quad x\in \Omega}.
\end{array}\right. 
\end{equation}

When the background density  $\lambda=1$, $\Lapl\Psi=\Delta\Psi$. The model has been investigated mathematically in this context. In \cite{BourgeoisBeale}, Bourgeois and Beale considered the inviscid case without boundary. Desjardins and Grenier studied in \cite{desjardins1998derivation} multiple cases using different approximations while considering the Ekman layers. The difference between viscous and non-viscous cases lies in the boundary equation.

Vasseur and Novack studied the global well-posedness of strong solutions in \cite{novack2018global}. Finally, Puel and Vasseur studied in \cite{puel2015global} the weak solutions for more general cases, where the background density parameter $\lambda(z)=-\frac{1}{\rho_z}$ verifies $\Lambda>\lambda>\frac{1}{\Lambda}$ for all $z>0$ and for all constant $\Lambda>1$. Novack improved the result for inviscid flows in \cite{novack_2019}.\\

This paper considers degenerate $\lambda=z^a$ and domains $\Omega$ as either periodic $\ts$ or bounded convex $C^1$ set. We study \eqref{3DQG} in the space $ \lbrace z>0\rbrace\times\Omega$ equipped with the following lateral boundary condition (A):
\begin{enumerate}
\item Periodic Case: when $\Omega=\mathbb{T}^2$, (A) denotes the periodic condition with integral $0$. That is, for a.e. $t\in\mathbb{R}_{+}$, $z\in \mathbb{R}_{+}$:
\[
\forall ~x \in \ts: \quad \Psi(t,z,x+2\pi \mathbf{e}_i)=\Psi(t,z,x); \qquad \int_{\ts}\Psi(t,z,x)dx=0,
\]
where $\mathbf{e}_i$ are unit vectors in $x_i$ direction, $i=1,2$.
\item Bounded Case: when $\Omega$ is a bounded convex domain, (A) denotes the following Dirichlet boundary condition:
\begin{equation}
\label{dirichlet}
\Psi=0, \qquad\text{ on  } \partial\Omega\times\lbrace z\geqslant 0\rbrace.
\end{equation}
It is rigorously defined via the trace theorem for $H^1(\Omega)$ functions.
\end{enumerate}

Our main result is the following
\begin{theorem}[Main theorem]
\label{main}
Let $\Omega$ be either the $2 D$ torus or a bounded convex domain in $\rt$ and $\lambda=z^a$, $a<1$. Then for any $T>0$ and initial data $\Lapl\Psi_0\in  L^2(\btp)\cap L^\infty(\btp)$, there exists a weak solution $\Psi$ of the system \eqref{3DQG} in $[0,T]\times(\btp)$ verifying the lateral boundary condition (A) such that $\gradb{\perp}\Psi\in L^\infty([0,T],L^2(\btp))$. That is, for any test functions $\phi\in C^\infty_c(\mathbb{R}_{+}\times\Omega)$ and $\varphi\in C^\infty(\Omega)$ and a.e. $t$, 
\[
\left\{ \begin{array}{l}
\int_{\btp}\Lapl\Psi(t)\phi-\int_{\btp}\Lapl\Psi(0)\phi-\int_0^T\int_{\btp}\gradb{\perp}\Psi \cdot \gradb{}\phi\Lapl\Psi=0, \\~\\
\int_{\Omega}\theta(t)\varphi-\int_{\Omega}\theta(0)\varphi-\int_0^T\int_{\Omega}\gradb{\perp}\Psi(t,0,x)\cdot\gradb{}\varphi \theta = \int_0^T\int_{\Omega}\Psi(t,0,x)\overline{\Delta}\varphi,  \\~\\
{\Psi(0, z, x)=\Psi_{0}( z,x)},
\end{array}\right. 
\]
where $\theta=\glmd \Psi$ is the Neumann boundary condition at $z=0$ defined in $H^{\frac{1-a}{2-a}}(\Omega)$ via \cref{trace}. Moreover, for almost every time $t$, $$\norm{\dslam\Psi(t)}_{L^2(\btp)}\leqslant\norm{\dslam\Psi_0}_{L^2(\btp)},$$
$$
\norm{\Lapl \Psi}_{L^p(\rtp)}=\norm{\Lapl\Psi_0}_{L^\infty(\rtp)}\quad \forall ~~2\leqslant p\leqslant \infty,$$
and $\theta \in L^\infty([0,T], L^2(\Omega))\cap L^2([0,T],H^{1/(2-a)}(\Omega))$.
\end{theorem}

Since we are working on an elliptic operator that is degenerated or singular near the boundary $z=0$, we lose some elliptic regularities near the boundary. Note that the weighted operators $\Lapl$, $\dslam$, and $\nabla_\lambda$ do not commute with derivatives in $z$. This prevents us from using directly classical theories. Therefore we first establish new results on weighted trace propositions and wellposedness. 

Let us introduce the critical generalized Surface Quasi-geostrophic equations with viscosity(see \cite{miao2012global}, \cite{may2011global}, \cite{constantin_ignatova_2016})
\begin{equation}
\label{gsqg}
\partial_t\theta +\gradb{\perp}(-\overline{\Delta})^{-\frac{1-a}{2-a}}\theta\cdot\gradb{}\theta=-(-\overline{\Delta})^{\frac{1}{2-a}} \theta.
\end{equation}

The generalized SQG equation can be derived from the singular QG system in the following way: Consider initial values such that
\[
\Lapl \Psi\vert_{t=0}=0.\]
Then $\Lapl \Psi\equiv 0$ uniformly for all time. Similarly as Caffarelli and Silverstre's extension over $\mathbb{R}^n$ in \cite{caffarelli2007extension}, we define the harmonic extension of $f$ on $\Omega$ by:
\[\left\lbrace
\begin{array}{l}
\Lapl \Psi(z,x)=0, \qquad \text{for }x\in \Omega ,~ z>0,\\ 
\Psi(0,x)=f(x), \qquad \text{for } x\in \Omega,~z=0.
\end{array}\right.
\]
We will prove that the Neumann boundary condition is given by
\[
\theta:=-\lim _{z \rightarrow 0} \lambda \Psi_{z}(z, x) = (-\overline{\Delta})^{\frac{1-a}{2-a}} f(x).
\]
Therefore, $\theta$ is a solution to \eqref{gsqg}.\\

In the next section, we formally derive the Quasi-Geostrophic System from the scaled primitive equations with the Ekman layer near the boundary on the periodic domain. Section 3 illustrates the lateral boundary conditions with their physical meaning. In section 4, we introduce several preliminaries. We first construct Caffarelli and Silvestre type extensions on a bounded domain. We can get better regularity properties from weighted operators in newly defined spaces by constructing new trace propositions. We then obtain transport propositions to study the wellposedness of the interior system. In section 5, we work on a decoupled and regularized QG system. For this, we split the solution $\Psi$ into the boundary part $\Psi_1$ and the interior part $\Psi_2$ as:
\[
\left\{\begin{array} { l }
{ \Delta _ { \lambda } \Psi _ { 1 } = 0 }, \\
{ \glmd \Psi _ { 1 } = \glmd \Psi =\theta },
\end{array} \qquad \left\{\begin{array}{l}
\Delta_{\lambda} \Psi_{2}=\Delta_{\lambda} \Psi, \\
\glmd \Psi_{2}=0.
\end{array}\right.\right.
\]
For the boundary equation of $\theta$, we use the Galerkin method and construct $\Psi_1$ as the harmonic extension of $\theta$. For the interior part, we modify the velocity and study it as a transport system. Using the weighted trace propositions, we then couple the equations on $\Psi_2$(interior equations) with those on $\Psi_1$(boundary equations). In section 6, thanks to the Aubin-Lions Lemma, we obtain the weak solutions of the QG system at the limit of the regularized systems.

\section{Derivation of the Model}
This section is devoted to the formal derivation of the model \ref{3DQG} following the methodology of \cite{BourgeoisBeale}. For simplicity, we restrict ourselves to the periodic case (A) 1. taking only into account the boundary at $z=0$. We will discuss the boundary conditions in the next section. Therefore, we consider $(z,x_1,x_2)\in \btp$ where $x$ is a two-dimensional variable and $\Omega$ the two-dimensional periodical domain $\ts$. The starting point is the scaled Navier-Stokes equation in the Boussinesq approximation with Coriolis force, see \cite{pedlosky2013geophysical}  \cite{desjardins1998derivation}:

\begin{equation}
\label{QGBV1.1}
\varepsilon\frac{Du}{Dt}-v=-\phi_{x_1}+\varepsilon^2 \Delta u,
\end{equation}
\begin{equation}
\label{QGBV1.2}
\varepsilon\frac{Dv}{Dt}+u=-\phi_{x_2}+\varepsilon^2 \Delta v,
\end{equation}
\begin{equation}
\label{QGBV1.3}
\varepsilon\frac{Dw}{Dt}+\rho=-\phi_z+\varepsilon^2 \Delta w,
\end{equation}
\begin{equation}
\label{QGBV1.4}
\operatorname{div}\ \mathbf{u}=0,
\end{equation}
\begin{equation}
\label{QGBV1.5}
\varepsilon\dfrac{D\rho}{Dt}+w\overline{\rho}_z=0,
\end{equation}
where $\mathbf{u}=(u,v,w)$ and $\dfrac{D}{Dt}=\partial_t+\mathbf{u}\cdot\nabla$ . Here $\rho=\overline{\rho}(z)+\tilde{\rho}$, $\varphi=\overline{\varphi}(z)+\phi(z,x_1,x_2,t)$ are the density and the pressure of the fluid. The function $\overline{\rho}(z)$ is the given background density profile, which we assume to satisfy $\overline{\rho}_z<0$.  Therefore $\tilde{\rho}$ is the fluctuation from $\overline{\rho}$. And $\overline{\varphi}(z)$ is the corresponding potential to $\overline{\rho}(z)$. Thus $\overline{\varphi}_z=-\overline{\rho}(z)$. Similarly, we let $\phi$ represent the fluctuation from the background potential $\overline{\varphi}$.\\

We first derive the equation of the bulk. Assume the solutions $\mathbf{U}(\varepsilon)=(\mathbf{u}(\varepsilon),\phi(\varepsilon))$, and substitute the formal expansion $\mathbf{U}(\varepsilon)=\mathbf{U}^{(0)}+\varepsilon \mathbf{U}^{(1)}+\varepsilon^{2} \tilde{\mathbf{U}}(\varepsilon)$ into our scaled equations. For $\mathcal{O}(1)$ we have from \eqref{QGBV1.1} to \eqref{QGBV1.5}:
\[
v^{(0)}=\phi_{x_1}^{(0)},
\]
\[
u^{(0)}=-\phi_{x_2}^{(0)},
\]
\[
\rho^{(0)}=-\phi_{z}^{(0)},
\]
\[
w^{(0)}=0.
\]

We use the notation
\[
d_{g} \equiv \partial_{t}+\mathbf{u}^{(0)} \cdot \nabla=\partial_{t}+u^{(0)} \partial_{x_1}+v^{(0)} \partial_{x_2},
\]
for zero-order (geostrophic) material derivative. Then the equations of order $\mathcal{O}(\varepsilon)$ are
\[
d_{g} u^{(0)}-v^{(1)}=-\phi_{x_1}^{(1)},
\]
\[
d_{g} v^{(0)}+u^{(1)}=-\phi_{x_2}^{(1)},
\]
\[
\rho^{(1)}=-\phi_{z}^{(1)},
\]
and the last two equations give
\[
\nabla \cdot \mathbf{u}^{(1)}=0,
\]
\[
d_{g} \rho^{(0)}+w^{(1)} \overline{\rho}_{z}=0.
\]
Now let $\lambda=1/\bar{\rho}_z$. With the notation \eqref{dlamnotation}, we have
\[
d_g(\nabla_\lambda\phi^{(0)})=d_g(-\lambda\rho^{(0)}, v^{(0)}, -u^{(0)})^{T}=(-w^{(1)}, -u^{(1)}-\phi^{(1)}_{x_2}, -v^{(1)}+\phi^{(1)}_{x_1})^{T}={\rm curl}Q,\]
for some $Q$.
Denoting $\Psi=\phi^{(0)}$, we obtain the gradient form of the $3D$ Quasi-geostrophic system:
\[
(\partial_t+\overline{\nabla}^\perp\Psi\cdot\overline{\nabla})(\nabla_\lambda\Psi)={\rm curl} Q.\]
Taking the divergence of this equation gives the first equation of \eqref{3DQG}:
\[
\left(\partial_{t}+\overline{\nabla}^{\perp} \Psi \cdot \overline{\nabla}\right)(\overline{\Delta}\Psi+\left(z^a \Psi_{z}\right)_{z})=0, \qquad {t>0, \quad z>0}.
\]

Considering the Ekman pumping near the boundary, we obtain an equation on the boundary by looking for solutions $U^{\varepsilon}(t, z, x_1, x_2)$ of the form:
\[U^{\varepsilon}(t, z, x_1, x_2)=U^{0}( z, x_1, x_2,t)+u_{b}(t, x_1, x_2, \zeta),\]
with $\lim\limits_{\zeta\to\infty} U^b=0$ and the Dirichlet boundary condition 
\begin{equation}
\label{deriveDirichletCon}
U^\varepsilon=0\text{ at  }z=0.
\end{equation}
Therefore the boundary condition for $U_\varepsilon$ is $U_\varepsilon=0$. Plugging these expansions in \eqref{QGBV1.1} to \eqref{QGBV1.5} and keeping the predominant terms, we obtain the following system:
\[
-v_{b}=\partial_{\zeta \zeta}^{2} u_{b}+\partial_x\phi_{b},
\]
\[
u_{b}=\partial_{\zeta \zeta}^{2} v_{b}+\partial_y\phi_{b},
\]
\begin{equation}
\label{QGBV2.3}
0=\partial_\zeta\phi_b.
\end{equation}
Equation \eqref{QGBV2.3} shows that the pressure is constant in the boundary layer, a common property in fluid boundary layer theory. Together with the Dirichlet boundary \eqref{deriveDirichletCon}, it gives $\phi_b=0$. We solve the above ordinary differential equations,
\[\begin{aligned}-v_{b} &=\partial_{\zeta \zeta}^{2} u_{b}, \\ u_{b} &=\partial_{\zeta \zeta}^{2} v_{b} ,\end{aligned}
\]
with the boundary conditions
\[
u_b(t, 0, x_1, x_2)=-u^{0}(t, 0, x_1, x_2)\quad\text{and}\quad v_b(t, 0, x_1, x_2)=-v^{0}(t, 0, x_1, x_2)\quad \text{at }\zeta=0.
\]
Therefore we have
\[
\begin{array}{l}
{u_{b}(\zeta, x_1, x_2, t)=}\\
\qquad{-\exp \left(-\frac{\zeta}{\sqrt{2}}\right)\left(u^{0}(0, x_1, x_2, t) \cos \left(\frac{\zeta}{\sqrt{2}}\right)+v^{0}(0, x_1, x_2, t) \sin \left(\frac{\zeta}{\sqrt{2}}\right)\right)}, \\
 {v_{b}(\zeta, x_1, x_2, t)=}\\
\qquad{\exp \left(-\frac{\zeta}{\sqrt{2}}\right)\left(u^{0}(0, x_1, x_2, t) \sin \left(\frac{\zeta}{\sqrt{2}}\right)-v^{0}(0, x_1, x_2, t) \cos \left(\frac{\zeta}{\sqrt{2}}\right)\right)}.\end{array}
\]
Due to the divergence free condition $\partial_x u_b+\partial_y v_b+\partial_z w_b=\partial_x u_b+\partial_y v_b+\partial_\zeta w_b/\varepsilon=0$, we have
\[\begin{aligned} w_{b}(\zeta, x_1, x_2, t)=-\frac{\varepsilon}{\sqrt{2}} \exp \left(-\frac{\zeta}{\sqrt{2}}\right)\left(\partial_{x} v^{0}(0, x_1, x_2, t)-\partial_{y} u^{0}(0, x_1, x_2, t)\right) \\ \times \left(\sin \left(\frac{\zeta}{\sqrt{2}}\right)+\cos \left(\frac{\zeta}{\sqrt{2}}\right)\right). \end{aligned}
\]
Using both formal expansions in the bulk and at the boundary, $w=w^{(0)}+\varepsilon w^{(1)}=\varepsilon w^{(1)}$, and now we have $w=w^{(0)}+w_b=0$ at $z=0$. This gives $\varepsilon w^{(1)}=-w_b$. In other words, at $\zeta=z=0$,
\[
\begin{aligned}
\lim\limits_{z\rightarrow 0} w^{(1)}=-\frac{w_b}{\varepsilon}=\frac{1}{\sqrt{2}}(v^{(0)}_{x_1}-u^{(0)}_{x_2})\vert_{z=0}=\frac{1}{\sqrt{2}}\left(\Psi_{x_1 x_1}+\Psi_{x_2 x_2}\right)\vert_{z=0} \\ =\lim\limits_{z\rightarrow 0}-\frac{d_g\rho^{(0)}}{\overline{\rho}_z}=\lim\limits_{z\rightarrow 0}\frac{d_g(\partial_z\Psi)}{\overline{\rho}_z}.\end{aligned}
\]
Therefore we obtain the following equation defined in the boundary $z=0$:
\[
\left(\partial_t+\gradb{\perp}\Psi\cdot\overline{\nabla}\right)\glmd \Psi=\left(-\overline{\Delta}\right)\Psi\qquad z=0.
\]
This is the second equation of \eqref{3DQG}.

\section{Boundary Conditions}

In the previous section, we derive our model without considering the lateral boundary condition. However, we will study the system when $x$ belongs to a bounded domain $\Omega$. We consider $\Omega$ to be either a bounded convex domain in $\rt$ or a $2D$ periodic domain $\ts$. \\

For the bounded case, we denote $\overrightarrow{n}$ the normal vector at the boundary. The impermeability requires that the velocity in normal direction $\gradb{\perp} \Psi\cdot\overrightarrow{n}$ is zero on the lateral boundary $\partial \Omega\times\lbrace z>0\rbrace $. This implies that the steam function $\Psi$ is constant on each level of the lateral boundary, i.e., $\partial \Omega\times\lbrace z\rbrace $ for each $z>0$. Novack and Vasseur considered this boundary condition when the parameter $\lambda$ is bounded from above and below in \cite{novack_vasseur_2020classicalbdd} \cite{novack_vasseur_2020bdd}. Our study focuses on the difficulties of the singularities of $\lambda$ at $z=0$. In this case, we consider a slightly simplified lateral boundary condition (A) 2. where the stream function $\Psi=0$ on $\partial \Omega\times\lbrace z>0\rbrace $. This boundary condition is also discussed in \cite{constantin2018local} \cite{constantin_ignatova_2016} for Surface Quasi-Geostrophic equations.

When considering the periodic case, we assume that $\int_{\ts}\Psi_0=0$ for all $z>0$. Note that this property is formally preserved by the system, as shown in the following lemma.
\begin{lemma}

Consider a smooth initial potential vorticity $\Lapl\Psi$ in $L^2(\ttp)$ verifying that $\int_{\ts}\Lapl\Psi_0(z,x)=0$ for all $z>0$. Assume $\Psi\in C^\infty\cap L^2(0,T;\ttp)$ is a smooth solution to \eqref{3DQG}, then $\bar{\Psi}:=\Psi-\int_{\ts}\Psi$ is also a solution to \eqref{3DQG}.
\end{lemma}
\begin{proof}
By applying $\gradb{}$ to $\bar{\Psi}$, $$\gradb{}\bar{\Psi}=\gradb{}\Psi. $$
Integrating the equation over $\Omega$, we have that for all $z>0$:
\[
\int_{\ts}\partial_t \Lapl \Psi +\int_{\ts}\gradb{\perp}\Psi \cdot\gradb{}\Lapl\Psi=0.
\]
Thus $\partial_t\int_{\ts} \Lapl \Psi(t,z,x)=0$ for all $ t$. Therefore, $\int_{\ts}\Lapl \Psi = \int_{\ts}\Lapl \Psi_0 $ $= 0$ for all $t$. Since on periodic domain, $\int_{\ts}\overline{\Delta}\Psi =0$. Therefore,
\[
\partial_z\int_{\ts}\lambda\partial_z\Psi dx= \int_{\ts}\zlz \Psi dx= \int_{\ts}\zlz \Psi_0 dx.
\]

Moreover, on the boundary $z=0$, integrating over $\ts$, 
$$\partial_t\int_{\ts}\glmd \Psi dx+\int_{\ts}\gradb{\perp}\Psi\cdot\gradb{}\glmd \Psi dx=\int_{\ts}\overline{\Delta}\Psi(t,0,x) dx=0.
$$
Therefore, $\int_{\ts}\glmd\Psi =\int_{\ts}\glmd\Psi_0$. But for $\bar{\Psi}$, we have
\[
\Lapl\bar{\Psi}=\Lapl \Psi-\zlz \int_{\ts}\Psi dx=\Lapl \Psi- \int_{\ts}\zlz\Psi_0 dx,\]
and 
\[
\glmd \bar{\Psi}=\glmd \Psi -\int_{\ts}\glmd\Psi =\glmd \Psi-\int_{\ts}\glmd\Psi_0.
\]
Note that the later terms do not depend on $x$ or $t$. Therefore, $\bar{\Psi}$ also solves the system \eqref{3DQG}.
\end{proof}
Without loss of generality, we will construct solutions on the periodic domain verifying $\int_{\ts}\Psi=0$ for a.e. $t$ and $z$ as prescribed in the lateral boundary condition (A) 1.\\


\section{Preliminaries}

\subsection{Definition of spaces}
We use the notation $L^p\left(\btp\right)$ and $L^p(\Omega)$ for the Lebesgue spaces. We denote the Sobolev spaces with integer number $s$ by $H^s\left(\btp\right)$ or $H^s\left(\Omega\right)$. \\

Let us define a basis for $L^2$ on either the $2D$ torus or the $2D$ bounded domain with the lateral boundary condition (A) considered. Let $-\overline{\Delta}_{\Omega}$ be the homogeneous Laplace operator defined on the corresponding domain, and $\lbrace e_i\rbrace_{i}$ be its orthonormal eigenfunctions. In other words, for each $i$,

\begin{gather}
-\overline{\Delta}_\Omega e_i(x)= k_i e_i(x),\label{defeigenv}\\
\text{with either  }\int_{\Omega} e_i=0 \text{ when }\Omega=\ts,\text{  or  } e_i\vert_{\partial\Omega}=0 \text{ when }\Omega\in\rt.
\label{defeigenc}
\end{gather}
Note that all eigenvalues $ k_i$ are positive. We also define the operator $\Lambda:=\sqrt{-\overline{\Delta}_{\Omega}}$, i.e. for all $i$:
\[
\Lambda e_i=\sqrt{k_i}e_i.
\]
If $f\in L^2$, then $f=\sum_{i=1}^\infty f_i e_i$ and $\norm{f}_{L^2(\Omega)}=\norm{\lbrace f_i\rbrace_i}_{l^2}$. If $f\in H^1(\Omega)$ and $f= \sum_{i=1}^\infty f_i e_i$, $\norm{f}_{H^1(\Omega)}=\norm{\lbrace f_i\rbrace_i}_{l^2}+\norm{\lbrace \sqrt{k_i} f_i\rbrace_i}_{l^2}$.\\

Let us define the homogeneous fractional Sobolev space with lateral boundary conditions using the basis $\lbrace e_i\rbrace^\infty_{i=1}$. Let $\theta=\sum_{i=1}^\infty \theta_i e_i$, where $\theta_i =\int_\Omega \theta(x) e_i(x) dx$. For$s\in(-1,1)$, we say that $\theta\in \dot{H}^s$ if $\lbrace  k_i^{s/2}\theta_i\rbrace_i$ lies in $l^2$. We define the norm as $\norm{\theta}_{\dot{H}^s(\Omega)} =(\sum^\infty_{k=1} k_i^s\theta_i^2)^{1/2}$. By definition $\dot{H}^s$ is a Banach space and $\dot{H}^s$ is the dual space of $\dot{H}^{-s}$, since for all $\theta_1\in \dot{H}^s$ and $\theta_2\in \dot{H}^{-s}$:
\begin{align*}
\int_{\Omega} \theta_1\theta_2 &=\sum_i\sum_j\int_{\Omega} \theta_{1,i}\theta_{2,j}e_ie_jdx,\\
& =\sum_i\theta_{1,i}\theta_{2,i}=\sum_i\theta_{1,i} k_i^{s/2}\theta_{2,k} k_i^{-s/2},\\
&\leqslant\norm{\lbrace k_i^{s/2}\theta_{1,k}\rbrace_i}_{l^2}\norm{\lbrace k_i^{-s/2}\theta_{2,i}\rbrace_i}_{l^2}=\norm{\theta_1}_{\dot{H}^{s}(\Omega)}\norm{\theta_2}_{\dot{H}^{-s}(\Omega)}.
\end{align*}

\begin{definition}
\label{projection}
For $u\in L^2(\Omega) $ or $L^2(\btp)$, we denote
\[ u_i:=\int_{\Omega} u e_i dx ,\qquad \PP n u := \sum_{i=1}^n u_i e_i.
\]
$\PP n$ is a projection that either maps from $L^2(\Omega)$ to $C^\infty(\Omega)$ or from $L^2(\btp)$ to $L^2_z(C^\infty(\btp))$.
\end{definition} 
\begin{remark}
As $n\to \infty$, $\PP n u \rightarrow u$ strongly in $L^2(\Omega)$ or $L^2(\btp)$. Furthermore, if $u\in H^1(\Omega)$, $\PP n u \to u $ in $H^1(\Omega)$ strongly too.
\end{remark}
\begin{lemma}
\label{equiv}
Fix a positive integer $n$ and a real number $s\in(0,1)$. Then for any function $u\in L^2(\Omega)$, $\norm{\PP n u}_{L^2(\Omega)}$, $\norm{\PP n u}_{\dot{H}^{s}(\Omega)}$ and $ \norm{\PP n u}_{\dot{H}^{-s}(\Omega)}$ are equivalent up to constants that only depend on $n$ and $s$. 
\end{lemma}
\begin{proof}
By definition $\norm{\PP n u}^2_{L^2(\Omega)}=\sum_{i=1}^n u_i^2$, $$\norm{\PP n u}^2_{\dot{H}^{s}(\Omega)}=\sum_{i=1}^n k_i^s u_i^2,$$ and  $$\norm{\PP n u}^2_{\dot{H}^{-s}(\Omega)}=\sum_{i=1}^n k_i^{-s} u_i^2.$$ 
So $k_1^s\norm{\PP n u}_{L^2(\Omega)}\leqslant \norm{\PP n u}_{\dot{H}^{s}(\Omega)}\leqslant k_n^s\norm{\PP n u}_{L^2(\Omega)}$. \end{proof}
\begin{definition}
Using the notation $\dslam=(\slam \partial_z, \gradb{})$, define the space $\mathcal{H}$ as
\[
\mathcal{H}=\lbrace \Psi\vert \dslam \Psi\in L^2(\btp), \Psi \text{ satisfies the condition (A)} \rbrace ,
\]
with the norm:
\[
\norm{\Psi}_{\hh}=\norm{\Psi}_{L^2(\Omega)}+\norm{\dslam\Psi}_{L^2(\btp)}.
\]
\end{definition}
We can equip the space $\mathcal{H}$ with the inner product :
\[
<\Psi,\phi>_{\hh}=\int_{\btp}\Psi\phi+\dslam\Psi\cdot\dslam\phi ~dxdz.
\]
Since with lateral boundary condition $(A)$ and Poincar\'e 's inequality,
\[
\norm{\Psi}_{L^2(\btp)}\leqslant C_p\norm{\dslam\Psi}_{L^2(\btp)},
\]
and the semi-norm $\norm{\dslam\Psi}_{L^2(\btp)}$ is equivalent to $\norm{\Psi}_{\hh}$. The constant $C_p$ from the Poincar\'e 's inequality only depends on $\Omega$. Before defining another space, let us define the Dirichlet trace operator $\gamma_0$. If $\Psi$ is a continuous function in $z$,
\[
\gamma_0\Psi(t,x)=\lim_{z\rightarrow 0}\Psi(t,z,x).\]

We introduce the following definition.

\begin{definition}
Define the space $\hhh$ by
\begin{equation}
\label{hhh}
\hhh=\hh \cap \lbrace u:\Lapl u \in L^2(\btp)\rbrace,
\end{equation}
with the norm:
\[
\norm{u}_{\hhh}=\norm{u}_{\hh}+\norm{\Lapl u}_{L^2(\btp)}.
\]
We also define two sets $\dc$ and $\ddd$ as
\[
\dc:=\lbrace u \in C^0(\mathbb{R}_{+};C^\infty(\Omega)\rbrace \cap \hh.
\]
\[
\ddd:=\lbrace \lambda\partial_z u \in C^0(\mathbb{R}_{+};C^\infty(\Omega)\rbrace \cap \hhh.
\]

Note that for any $u\in \dc$, $\gamma_0 u(x)= u(0,x)$ and for any $u\in \ddd$, $\gamma_{\lambda} u(x)=\glmd  u(0,x)$.
\end{definition}

\subsection{Extension on bounded domains}
Similar to the extension constructed on $\mathbb{R}^n$ by Caffarelli Silverstre \cite{caffarelli2007extension}, we now construct the extension from a 2D bounded domain $\Omega$ (either $2 D$ Torus or bounded convex domain in $\rt$) to $\Omega\times\mathbb{R}_{+}$.  
\begin{definition}
For any $\theta\in H^{-\frac{1-a}{2-a}}(\Omega)$, we say that $\Psi_1$ is a Neumann extension of $\theta$ in $\btp$ if
\begin{align}
\label{eqn:extension1}
\begin{cases}
\Lapl\Psi_1=0, & \text{ in }\btp,\\
\glmd\Psi_1=\theta,
\end{cases}
\end{align}
with lateral boundary (A). Similarly, for any $h\in H^{\frac{1-a}{2-a}}(\Omega)$, we say that $\Psi_2$ is a Dirichlet extension of $h$ in $\btp$ if 
\begin{align}
\label{eqn:extension2}
\begin{cases}
\Lapl\Psi_2=0, & \text{ in }\btp,\\
\gamma_0\Psi_2=h,
\end{cases}
\end{align}
with lateral boundary (A).
\end{definition}
We first study such extensions for specific boundary conditions at $z=0$.
\begin{lemma}
\label{ExtensionLemma}
For any $a<1$, there exist constants $\kappa >0$, $C_a>0$, $ \psi_i \in C^0_z C_x^\infty$ and $\lambda\partial_z\psi_i\in C^0_zC^\infty_x$ such that for each $i$,
\begin{equation}
\label{Psii}
\begin{cases}
\Lapl \psi_i=0,\\
\psi_i\vert_{z=0}= k_i^{-\frac{1-a}{2(2-a)}}e_i(x).
\end{cases}
\end{equation}
Moreover, for $i,j\in\mathbb{N}^*$ 
\begin{equation}
\label{eigenfunction}
\int_0^\infty\int_\Omega\dslam \psi_i\dslam\psi_j dxdz=\kappa\delta_{ij},
\end{equation}

\begin{equation}
\label{eigenNuemann}
\lim\limits_{z\rightarrow 0}z^a\partial_z\psi_i = C_a  k_i^{\frac{1-a}{2(2-a)}} e_i.
\end{equation}
\end{lemma}
\begin{proof}
To construct our solution $\psi_i= k_i^{-\frac{1-a}{2(2-a)}}e_i(x)Z_i(z)$, for some function $Z_i(z)$ verifying the following equation
\begin{equation}
\label{Zequation}
\begin{cases}
\partial_z(z^a\partial_z Z_i)= k_i Z_i & z>0,\\
Z_i(0)=1 \qquad \lim\limits_{z\rightarrow \infty}Z_i(+\infty)=0.
\end{cases}
\end{equation}
Note that, since $e_i$ are eigenfunctions, this implies
\begin{align*}
\Lapl \psi_i= k_i^{-\frac{1-a}{2(2-a)}}\left(\overline{\Delta}e_i(x)\cdot Z_i(z)+e_i(x)\partial_z\lambda\partial_z Z_i(z)\right),\\
= k_i^{-\frac{1-a}{2(2-a)}}e_i(x)\left(- k_i \cdot Z_i(z)+\partial_z z^a\partial_z Z_i(z)\right)=0.
\end{align*}

To construct a solution to \eqref{Zequation}, let us first consider $W(w)$ solution to the following ODE:
\begin{equation}
\label{ODE}
W(w)=w^{-a/(1-a)} W''(w),
\end{equation}
with $W(0)=1$ and $\lim\limits_{w\rightarrow \infty}W(w)=0$. We notice that for $\delta$ small enough, then 
\[
\overline{W}(w):=\min(1,w^{-\delta}),
\]
is a super solution, whereas for A and B large
\[
\underline{W}(w)\coloneqq\exp(-Aw^{1/2}-B w^2),
\]
is a subsolution. Thus by Perron's method \cite{caffarelli2007extension}, we find the solution $W(w)$ in between that solves \eqref{ODE}. Especially, $0\leqslant W(w)\leqslant 1$ since it is trapped between $\underline{W}$ and $\overline{W}$. From the ODE \eqref{ODE}, we also see that $0\leqslant W''(w)\leqslant w^{a/(1-a)}\in L^1 ([0,1])$, since $\frac{a}{1-a}>1$. Therefore $W'\in C([0,1])$. But $W$ is convex and $\lim\limits_{w\rightarrow +\infty} W(w)=0$, so $W'(0)<0$.

Note that $k_i>0$ and $1-a>0$, solutions to \eqref{Zequation} are constructed using the change of variables $w= k_i^{(1-a)/(2-a)}z^{1-a}$ and $Z_i(z)=W(w)$. Therefore $Z_i$ verifies:
\begin{align*}
&\qquad- k_i Z_i+\partial_z(z^a\partial_z Z_i)\\
&=- k_i W(w)+ k_i^{1/(2-a)}w^{-a/1-a}\partial_w  \left( k_i^{-a/(2-a)}w^{a/1-a} k_i^{1/(2-a)}w^{-a/1-a}\partial_w W(w)\right),\\
&= k_i\left(-W(w)+ w^{\frac{-a}{1-a}}\partial_{ww}W(w)\right)=0,
\end{align*}
with $Z_i(0)=W(0)=1$ and $\lim\limits_{z\rightarrow \infty}Z_i(z)=\lim\limits_{w\rightarrow \infty}W(w)=0$.

Moreover, let $J(W)=\int_0^\infty w^{\frac{a}{1-a}} W^2+W'(w)^2 dw=\int_0^\infty W W''+W'(w)^2 dw=WW'\vert^\infty_0$. $W(0)W'(0)=1\times W'(0)=W'(0)$ is a negative constant. On the other hand, since $1\geqslant\lim\limits_{w\rightarrow +\infty} \dfrac{ W^2}{ w^{-2\delta}}=\lim\limits_{w\rightarrow +\infty}\dfrac{2 W'W}{-2\delta w^{-2\delta-1}} $, since for $\lim\limits_{w\rightarrow +\infty} w^{-2\delta-1}= 0$, $\lim\limits_{w\rightarrow +\infty} WW'(w)=0$. Thus $J(W)$ is finite.

We construct our solution $\psi_i= k_i^{-\frac{1-a}{2(2-a)}} e_i(x)Z_i(z)$. Note that 
\[
\lim\limits_{z\rightarrow 0}z^a\partial_z\psi_i= k_i^{\frac{1-a}{2(2-a)}}\left(\lim\limits_{z\rightarrow 0} k_i^{-\frac{1-a}{2-a}}z^a\partial_zZ(z)\right)e_i= k_i^{\frac{1-a}{2(2-a)}}\left(\lim\limits_{w\rightarrow 0}W'(w)\right)e_i,\]
and 
\[
\psi_i\vert_{z=0}= k_i^{-\frac{1-a}{2(2-a)}}e_i.
\]
Therefore $C_a=\lim\limits_{w\rightarrow 0}W'(w)$. When $i\neq j$, the semi-inner product
\begin{align*}
&\qquad\int_0^\infty\int_\Omega\dslam \psi_i \dslam \Psi_j \\
& = k_i^{-\frac{1-a}{2(2-a)}}k_j^{-\frac{1-a}{2(2-a)}}\left(\int_0^\infty\int_\Omega \gradb{}e_i\gradb{}e_jdx Z_{i}Z{j}dz+\int_0^\infty\int_\Omega e_ie_jdx z^aZ'_{i}Z'{j}dz\right),\\
&=0.
\end{align*}
When $i=j$, the semi-inner product
\begin{align*}
\int_0^\infty\int_\Omega\abs{\dslam \psi_i }^2 & = k_i^{-\frac{1-a}{2-a}}\left(\int_0^\infty\int_\Omega (\gradb{}e_i)^2dx (Z_{i})^2dz+\int_0^\infty\int_\Omega e_i^2dx z^a(\partial_z Z_{i})^2dz\right),\\
&=J(W)\norm{e_i}_{L^2(\Omega)}^2,
\end{align*}
since $e_i$ is normalized in $H^1_0(\Omega)$. We finish our proof by letting $\kappa=J(W)$.\\

It is obvious that $ \Psi_j \in C^0_z C_x^\infty$. Since $W'(w)$ is continuous on $[0,+\infty)$,\\ $\lambda\partial_z\Psi_j(z,x)=z^a Z'_j(z)= k_i^{-\frac{1-a}{2-a}}W'(w)$ is continuous near $0$. Therefore\\ $\lambda\partial_z\Psi_j\in C^0_zC^\infty_x$.
\end{proof}

\begin{lemma}
\label{Hnorm}
$\mathcal{H}$ is a Banach space. Moreover, for every $\psi_i$ constructed in\\ \cref{ExtensionLemma}, $\norm{\psi_i}_{\mathcal{H}}=\kappa^{1/2}\norm{e_i}_{L^2(\Omega)}=\kappa^{1/2}$. 
\end{lemma}
\begin{proof}
We only prove that the space $\mathcal{H}$ is complete. 
Let us define the space $\tilde{H}=\lbrace \norm{\dslam u}_{L^2}+\norm{u}_{L^2}<+\infty\rbrace$. Note that $\hh$ is a subspace of $\tilde{H}$. The norm of the space $\mathcal{H}$ is defined as the same as $\norm{u}_{\tilde{H}}=\norm{\dslam u}_{L^2(\btp)}+\norm{ u}_{L^2(\btp)}$. When considering the periodic case, we define a operator $T_{\ts} u = \int_{\ts} u dx$, mapping from $\tilde{H}~\rightarrow L^2(0,\infty)$. Since $\int_{0}^\infty[\int_{\ts} u dx]^2dz\leqslant\norm{u}_{L^2(\ttp)}$, $T_{\ts}$ is a bounded and continuous linear operator. When considering the bounded domain in $\rt$, we  define the $T_{\partial\Omega} u$ the trace operator as $T_{\partial \Omega}u =u\vert_{\partial\Omega}$ from $\tilde{H}$ to $H^{1/2}(\partial\Omega\times \lbrace z>0\rbrace)$. Let $u_N$ be a Cauchy sequence in $\mathcal{H}$. Therefore, $u_N$ converges to $u$ in $\tilde{H}$ since $\tilde{H}$ is a Banach space with the same norm. Note that the operators $T_{\ts}$ and $T_{\partial\Omega}$ are continuous for both cases, and the limit $u$ is also in the space $\mathcal{H}$. Therefore, $\mathcal{H}$ is a Banach space. 
\end{proof}
\begin{proposition}[Extension operators]
\label{PExtensionLemma}
There exist bounded operators $E_1:H^{-\frac{1-a}{2-a}}(\Omega)\rightarrow \hh$ and $E_2:H^{\frac{1-a}{2-a}}(\Omega)\rightarrow \hh$ defined such that for any $\theta \in H^{-\frac{1-a}{2-a}}(\Omega)$, $E_1(\theta)=\Psi_1$ is the unique Neumann extension of $\theta$ verifying \eqref{eqn:extension1}, and for any $h \in H^{\frac{1-a}{2-a}}(\Omega)$, $E_2(h)=\Psi_2$ is the unique Dirichlet extension of $h$ verifying \eqref{eqn:extension2}. Moreover, $\Psi_1$ and $\Psi_2$ verify:
\[
\iint_{\btp}\abs{\dslam \Psi_1}^2dxdz=\norm{\theta}^2_{\dot{H}^{-\frac{1-a}{2-a}}(\Omega)},\]
and 
\[
\iint_{\btp}\abs{\dslam\Psi_2}^2dxdz=\norm{h}^2_{\dot{H}^{\frac{1-a}{2-a}}}(\Omega).\]
We call $E_1$ the Neumann extension operator and $E_2$ the Dirichlet extension operator.
\end{proposition}
\begin{proof}
Since $\theta \in \dot{H}^{-\frac{1-a}{2(2-a)}}(\Omega)$ and $h \in \dot{H}^{\frac{1-a}{2-a}}(\Omega)$, using the eigenfunctions defined previously, we can write that $\theta =\sum_{i=1}^\infty \theta_i  k_i^{\frac{1-a}{2(2-a)}} e_i$ and $h=\sum_{i=1}^\infty h_i  k_i^{-\frac{1-a}{2(2-a)}} e_i$. Now let us take $\theta_N =\sum_{i=1}^N \theta_i  k_i^{\frac{1-a}{2(2-a)}} e_i$ and $h_N=\sum_{i=1}^N h_i  k_i^{-\frac{1-a}{2(2-a)}} e_i$. From \cref{ExtensionLemma}, using the Dirichlet extension of eigenfunction $\psi_i$ we have that
\[
\Psi_{1,N}=\sum^N_{i=1} \theta_i\psi_i, \qquad\Psi^{2,N}=\sum_{i=1}^N h_i \psi_i.
\]
Here $\Psi_{1,N},~ \Psi_{2,N}$ verify:
\[
\Lapl\Psi_{1,N}=\Lapl\Psi_{2,N}=0\qquad\text{ in }\btp,\]
and lateral boundary condition (A). Moreover at $z=0$,
\[
\Psi_{1,N}(0, x)=\sum_{i=1}^N \theta_i  k_i^{\frac{1-a}{2(2-a)}} e_i \rightarrow \sum_{i=1}^\infty \theta_i  k_i^{\frac{1-a}{2(2-a)}} e_i=\theta(x) \text{ uniformly in } \dot{H}^{-\frac{1-a}{2-a}}(\Omega),\]
and 
\[
\Psi_{2,N}(0, x)=\sum_{i=1}^N h_i  k_i^{-\frac{1-a}{2(2-a)}} e_i \rightarrow \sum_{i=1}^\infty h_i  k_i^{-\frac{1-a}{2(2-a)}} e_i=h(x) \text{ uniformly in } \dot{H}^{\frac{1-a}{2-a}}(\Omega).\]
From the \cref{ExtensionLemma}, we have that
\[
\iint_{\btp}\abs{\dslam \Psi_{1,N}}^2 dxdz =\norm{\theta_N}^2_{\dot{H}^{-\frac{1-a}{2-a}}(\Omega)},
\]
\[
\iint_{\btp}\abs{\dslam \Psi_{2,N}}^2 dxdz =\norm{h_N}^2_{\dot{H}^{\frac{1-a}{2-a}}(\Omega)}.
\]
Using the fact that the $\lbrace\psi_i\rbrace_{i}$ are orthogonal for the semi-inner product,
\begin{align*}
&\qquad \iint_{\btp}\left( \dslam \Psi_{1,N}\dslam \Psi_{2,N}\right)dxdz\\
=&\sum_{i=1}^N\sum_{i=1}^N\iint_{\btp} \theta_i  k_i^{\frac{1-a}{2(2-a)}} h_jk_j^{-\frac{1-a}{2(2-a)}} \dslam \psi_i\dslam \psi_j,\\
=&\kappa\sum_{i=1}^N \theta_i h_i = \kappa \int_{\Omega}\theta_N h_N dx.
\end{align*}
Moreover, letting $N\rightarrow\infty$, $\Psi_{1,N},~\Psi_{2,N}$ converge uniformly to $\Psi_1:=E_1(\theta)$ and $\Psi_2:=E_2(h)$ in $\mathcal{H}$ due to the completeness of $\hh$.

To show the uniqueness of Dirichlet extension, let $\Psi_1$ and ${\Psi_1}'$ verifies $\Lapl\Psi_1=\Lapl{\Psi_1}'=0$ with lateral boundary condition (A) and $\gamma_0\Psi_1=\gamma_0{\Psi_1}'=h(x)$. Then
\begin{align*}
0=&\iint_{\btp}\Lapl(\Psi_1-{\Psi_1}')\cdot(\Psi_1-{\Psi_1}'),\\
=&\lim\limits_{N\rightarrow \infty}\iint_{\btp}\Lapl(\Psi_{1,N}-\Psi_{1',N}) (\Psi_{1,N}-\Psi_{1',N})dxdz,\\
=&-\lim\limits_{N\rightarrow \infty}\iint_{\btp}\abs{\dslam (\Psi_{1,N}-\Psi_{1',N})}^2 dxdz+\int_{\Omega}0\lim\limits_{z\rightarrow 0}(\Psi_{1,N}-\Psi_{1',N})dx,\\
=&-\norm{\Psi_1-{\Psi_1}'}_{\mathcal{H}}.
\end{align*}
Therefore, the extension operators $E_1$ and $E_2$ are well defined.
\end{proof}
\subsection{Singular Elliptic System}
In this section, we construct a solution to the elliptic system with the Neumann boundary condition $\glmd \Psi=0$ and lateral boundary conditions (A) and obtain its regularity. We will use this existence lemma to construct the interior solution in the bulk to \eqref{3DQG} for any time $t$ in section 5. In this part, $\Omega$ is either the $2D$ periodic domain $\ts$ or the bounded convex domain in $\rt$.

\begin{lemma}
\label{Lweightedexistence}

Let $\Omega=\ts$ or a bounded convex domain of $\rt$. We consider $\lambda=z^a$ for $a<1$, and $f\in L^2(\btp)$ with lateral boundary condition (A). Then there exists a unique solution $u\in \hhh$ to:
\begin{numcases}{}
-\Delta_\lambda u=f, & $z>0$, $\alpha\in\Omega$ \label{Einterior}\\
\glmd u =0. & ~
\end{numcases}

Moreover $\norm{\dslam u}_{L^2(\btp)}+\norm{\gradb{}\dslam u}_{L^2(\btp)}\leqslant\norm{f}_{L^2(\btp)}$. Hence we can define a linear bounded operator $G(f)=\dslam u$ mapping from $L^2(\btp)\rightarrow L^2(\btp)$ and this operator commutes with $\partial_t$ and $\gradb{}$ if $f\in C^1(C^\infty_x(\btp))$.  
\end{lemma}

\begin{proof}
Consider the following bi-linear form $B$ mapping from $\hh\times \hh \rightarrow \mathbb{F}$ and functional $F$ mapping from $\hh\rightarrow \mathbb{F}$,
\[
B(u,v)=\int_{\btp} \dslam u\dslam v dx dz, \qquad F(v)=\int_{\btp} fv dx dz.
\]
When we consider $\btp$, where $\Omega$ is $\ts$ or a bounded domain, we have that 
\[
\int_{\btp}\dslam v \cdot\dslam u\leqslant\norm{v}_{\mathcal{H}}\norm{u}_{\mathcal{H}},
\]
and 
\[
\int_{\btp}\dslam u \cdot\dslam u\geqslant\norm{u}^2_{\mathcal{H}}.
\]
Besides, we also have, 
\[
\abs{F(v)}\leqslant\norm{f}_{L^2}\norm{v}_{L^2}\leqslant\norm{f}_{L^2}\norm{v}_{\mathcal{H}}.
\]
Using the Lax-Milgram theorem, there exists a unique $u\in H$ to solve the variational problem:
\[
\int_{\btp}\dslam v \cdot\dslam u=\int_{\btp}f v,
\]
$\forall v \in \hh$ with $\norm{u}_{\hh}\leqslant\norm{f}_{L^2}$.

For all compactly supported test function $\phi\in C_c^\infty(\btp)$: $\phi\in \hh$ since
\[
\norm{\dslam \phi}_{L^2(\btp)}\leqslant\sup_{\supp \phi}\slam\norm{\nabla \phi}_{L^2(\btp)}.\]

Due to the Lax-Milgram theorem, there exists a unique $u\in \hh$ that solves the following equation:
\begin{equation}
\label{variational form}
B(u,\phi)=F(\phi).
\end{equation}
Therefore $u$ verifies the first equation \eqref{Einterior} in the sense of distribution. On the boundary, for any test function $\varphi\in C^\infty(\Omega)\cap\dot{H}^{\frac{1-a}{2-a}}$, using \cref{PExtensionLemma} there exists an extension $\tilde{\varphi}\in \hhh$ verifying:
\[
\begin{cases}
\Lapl \tilde{\varphi}=0& \text{ in }\btp,\\
\gamma_0\tilde{\varphi}(x)=\varphi(x) & \text{ at } z=0.
\end{cases}
\]
Therefore on the boundary $\lbrace z=0\rbrace\times \Omega$ and $\lbrace u_N\rbrace^\infty_{N=1}$ a sequence of functions in space $\ddd$ that converges to $u$ in space $\hhh$,
\[
\int_{\Omega} \gamma_{\lambda} u_N \varphi(x) = \int_{\btp}\nabla_\lambda u_N \nabla\tilde{\varphi}-\int_{\btp}-\Lapl u_N \tilde{\varphi}.
\]
Using equation \eqref{variational form}, the right-hand side converges to $0$ as $N\rightarrow \infty$. And the left-hand side converges to $\int_{\Omega}\glmd u \phi(x) dx$ since:
\[
\begin{array}{l}
\qquad\abs{\int_{\Omega}\glmd u  \varphi(x) dx- \gamma_{\lambda} u_N \varphi (x) }\leqslant\int_{\Omega}\abs{(\glmd u-\glmd u_N)\varphi} dx\\~\\
\leqslant \norm{\glmd u-\glmd u_N}_{\dot{H}^{-\frac{1-a}{2-a}}(\Omega)}\norm{\varphi}_{\dot{H}^{\frac{1-a}{2-a}}(\Omega)}\rightarrow 0,
\end{array}
\]
while $u_N$ converges to $u$ in $\hhh$. \\

Thus, there exists a unique $u\in \hh$ that solves the following system,
\[
\begin{cases}
\Lapl u=f, &\text{ in }\btp, \\
\glmd u =0. & \text{ on } \lbrace z=0\rbrace\times \Omega
\end{cases}
\]
due to Lax-Milgram. Moreover, $\norm{u}_{\mathcal{H}}\leqslant\norm{f}_{L^2(\btp)}$.\\

To obtain more regularity for the periodic case, let $v=-D^h_i D^h_i u$, where $D^h_i u=\dfrac{u(x+h\cdot e_i)-u(x)}{h}$, $k=1,2$ and $e_1$, $e_2$ are the unit vectors of the horizontal direction,
\[
\begin{array}{l}
~~~~-\int_{\ttp}\dslam u\dslam v\\
=-\int_{\ttp}\dslam u \dslam(D^{-h}_i D^h_i u)\\
=\int_{\ttp} D^h_i(\dslam u)\dslam(D^h_i u)\\
=\int_{\ttp} \abs{D^h_i(\dslam u)}^2.
\end{array}
\]
Thus we have
\[
\int_{\ttp}\abs{D^h_i(\dslam u)}^2\leqslant \int_{\ttp} fv,
\]
while
\[
\begin{array}{rl}
\abs{\int_{\ttp} fv} & \leqslant \int_{\ttp}\abs{f D^{-h}_i D^h_i u}\\
& \leqslant \norm{f}_{L^2}\norm{D^h_i\gradb{} u}_{L^2}
\end{array}
\]
Thus $\norm{D^h_i (\dslam u)}_{L^2}\leqslant \norm{f}_{L^2}$ uniform for all $h$. Thus, we have $\norm{\gradb{}\dslam u}_{L^2}\leqslant \norm{f}_{L^2}$.

The proof above only works for the periodic cases. We present now an alternative proof that works also for the bounded case. The proof is split into several steps.
\begin{steps}
\item We prove that for any $u\in H^1(\Omega)$, we have that
\[
\int_\Omega \abs{\nabla u}^2 =\int_{\Omega}\abs{\Lambda u}^2,
\]
where $\Lambda u\coloneqq(-\overline{\Delta})^{1/2}u =\sum_{i=1}^\infty u_i\sqrt{k_i}e_i$, where $u=\sum_{i=1}^\infty u_ie_i$. Here $\lbrace e_i\rbrace_{i=1}^\infty$ are eigenfunctions defined in \eqref{defeigenv}. Let $e_i$ and $e_j$ be two eigenfunctions and note that these functions are smooth in the interior of $\Omega$ and equipped with either periodic or Dirichlet boundary, so we can apply the divergence theorem on $\Omega$ and find 
\[
\int_{\Omega} \overline{\nabla} e_{i} \cdot \overline{\nabla} e_{j}=-\int_{\Omega} e_{i} \overline{\Delta} e_{j}=k_{j} \int_{\Omega} e_{i} e_{j}=k_{j} \delta_{ij}.
\]
Consider a function $u=\sum u_{i} e_{i}$ which is an element of $H^{1}$, by which we mean $\sum k_{i} u_{i}^{2}<\infty$. Since $\left\|\nabla e_{i}\right\|_{L^{2}(\Omega)}=\sqrt{k_{i}}$, the following sums all converge in $L^{2}(\Omega)$ and hence the calculation is justified:

\begin{align}
\int_{\Omega}|\overline{\nabla}  u|^{2} &=\int_{\Omega}\left(\sum_{i} u_{i} \overline{\nabla}  e_{i}\right)\left(\sum_{j} u_{j} \overline{\nabla}  e_{j}\right),\nonumber  \\
&=\int_{\Omega} \sum_{i, j}\left(u_{i} u_{j}\right) \overline{\nabla}  e_{i} \cdot \overline{\nabla}  e_{j}, \nonumber \\
&=\sum_{i, j}\left(u_{i} u_{j}\right) \int_{\Omega} \overline{\nabla}  e_{i} \cdot \overline{\nabla}  e_{j}, \nonumber \\
&=\sum_{j} k_{j} u_{j}^{2}=\int_{\Omega}\abs{\Lambda u}^2. \label{changelambdatogradient}
\end{align}

\item 
For fixed $N$ and any $f\in L^2$, let $f_N=\sum_{j=1}^N\varphi_j(z)e_j(x)\in L^2$ where $ \varphi_j(z)=\int_{\btp}f(z,x)e_j(x) dx$. We will show that the solution $u_N$ to:
\begin{numcases}{}
-\Delta_\lambda u_N=f_N & $z>0$ \label{EPinterior}\\
\glmd u_N =0 & ~	\label{EPboundary}
\end{numcases}
exist. Moreover $\norm{\overline{\nabla} \dslam u_N}^2_{L^2(\btp)}\leqslant\norm{f_N}^2_{L^2}\leqslant\norm{f}^2_{L^2}$ for all $N$.\\

Note that for each $j=1,..., N$, 
\[
\Lapl e_j=\overline{\Delta} e_j= k_j e_j.
\]
The solution $u_N$ can be written as $u_N=\sum_{j=1}^N \psi_j(z)e_j(x)$ and for each $j=1,..., N$, where $\psi_j(z)$ solves the following ODE,
\begin{equation}
\label{projectedode}
\partial_z\lambda\partial_z\psi_j(z)+k_j\psi_j(z)=\phi_j(z),
\end{equation}
with boundary condition $\glmd \psi_j(z)=0$.
Lax-Milgram theorem ensures the existence and uniqueness of $\psi_i$ for $j=1,..., N$ solution to \eqref{projectedode}. Moreover, $\Psi_j$ verifies 
\[
\sqrt{\lambda}\psi_j'(z)\in L^2([0,\infty)), ~~\sqrt{k_j}\psi_j(z)\in L^2([0,\infty)).
\]
Applying $\Lambda$ to \eqref{EPinterior}, we have 
\begin{align*}
\Lambda\Lapl u_N&=\Lambda f_n.\\
\text{i.e. for each }j\quad\left(\partial_z\lambda\partial_z\psi_j(z)+k_j\psi_j(z)\right)\sqrt{k_j}~ e_j(x)&=\varphi_j(z)\sqrt{k_j}~ e_j(x).
\end{align*}

Let us multiply this equation by $\psi_j(z)\sqrt{k_j}e_j$ and integrate on $\btp$. After integration by parts, we have:
\begin{align*}
\int_0^\infty\left(\partial_z\lambda\partial_z\psi_j(z)+k_j\psi_j(z)\right)\psi_j(z) dz \int_{\Omega} \sqrt{k_j}~ e_j(x)\sqrt{k_j}~ e_j(x)dx\\=\int_0^\infty\varphi_j(z)\psi_j(z) dz\int_{\Omega}\sqrt{k_j}~ e_j(x)\sqrt{k_j}~ e_j(x),\\
\end{align*}

\begin{align*}
-\int_0^\infty\left(\sqrt{\lambda}\partial_z\psi_j(z)\right)^2\cdot k_j\cdot\int_{\Omega} e_j(x)^2dx -\int_0^\infty\psi_j(z)^2 dz \int_{\Omega} (\gradb{}\Lambda e_j(x))^2dx \\
=\int_0^\infty\varphi_j(z)\psi_j(z) dz\cdot\int_{\Omega} e_j(x)(-\overline{\Delta})e_j(x) dx.
\end{align*}
Accordingly,
\begin{align*}
&\qquad\norm{\sqrt{\lambda}\partial_z\psi_j(z)}_{L^2}^2\norm{\Lambda e_j}^2_{L^2(\Omega}+\norm{\psi_j}_{L^2}^2\norm{\gradb{} \Lambda e_j(x)}_{L^2(\Omega)}^2 \\
&= \abs{\int_{\btp}\left(\varphi_j(z)e_j(x)\right)\cdot\left(\psi_j(z) (-\overline{\Delta})e_j(x)\right) dxdz}. 
\end{align*}
Especially,
\begin{align*}
&\norm{\psi_j}_{L^2}^2\norm{(-\overline{\Delta}) e_j(x)}_{L^2(\Omega)}^2  \leqslant \abs{\int_{\btp}\left(\varphi_j(z)e_j(x)\right)\cdot\left(\psi_j(z) (-\overline{\Delta})e_j(x)\right) dxdz} ,\\
&\qquad \leqslant \frac{1}{2}\norm{\psi_j}_{L^2}^2\norm{(-\overline{\Delta}) e_j(x)}_{L^2(\Omega)}^2  +\frac{1}{2}\int_0^\infty\left(\varphi_j(z)e_j(x)\right)^2.
\end{align*}

Therefore, for each $j$,
\begin{align*}
&\qquad\norm{\psi_j}_{L^2}^2\norm{(-\overline{\Delta}) e_j(x)}_{L^2(\Omega)}^2  +\int_0^\infty\left(\varphi_j(z)e_j(x)\right)^2\\
\geqslant & \abs{\int_{\btp}\left(\varphi_j(z)e_j(x)\right)\cdot\left(\psi_j(z) (-\overline{\Delta})e_j(x)\right) dxdz},\\
\geqslant & \norm{\sqrt{\lambda}\partial_z\psi_j(z)}_{L^2}^2\norm{\Lambda e_j}^2_{L^2(\Omega}+\norm{\psi_j}_{L^2}^2\norm{\gradb{} \Lambda e_j(x)}_{L^2(\Omega)}^2 \\
&+\norm{\psi_j}_{L^2}^2\norm{(-\overline{\Delta}) e_j(x)}_{L^2(\Omega)}^2,
\end{align*}
Therefore,
\begin{align*}
\int_0^\infty\left(\varphi_j(z)e_j(x)\right)^2 & \geqslant \norm{\sqrt{\lambda}\partial_z\psi_j(z)}_{L^2}^2\norm{\Lambda e_j}^2_{L^2(\Omega}+\norm{\psi_j}_{L^2}^2\norm{\gradb{} \Lambda e_j(x)}_{L^2(\Omega)}^2,\\
&=\norm{\dslam \Lambda(\psi_j(z)e_j(x))}_{L^2(\btp)}.
\end{align*}
Adding up all the inequalities for $j=1,2,...,N$,
\begin{equation}
\norm{f_N}^2_{L^2}\geqslant \norm{\overline{\nabla} \dslam u_N}^2_{L^2(\btp)}.
\label{ProjRegularity}
\end{equation}

Note that $f_N$ converges to $f$ strongly in $L^2$, and using Lax-Milgram, for all $N$, $\dslam u_N$ converges weakly in $L^2$. Therefore $u_N$ also converges weakly to some $u$ in $\hh$. Then in distribution sense $\gradb{}\dslam u_N$ converges to $\gradb{}\dslam u$. So using \eqref{ProjRegularity} we get $\norm{\gradb{}\dslam u}_{L^2(\btp)}\leqslant \norm{f}_{L^2(\btp)}$.
\end{steps}

We have already shown that the operator $G$ is linear bounded and well-defined. If $f\in C^1_t(C^\infty_x(\btp))$, there exist a solution $\tilde{\Psi}\in L^2(\btp)$ such that $\dslam\tilde{\Psi}=G(\partial_t f)$ and $\tilde{\Psi}$. Therefore, $\partial_t\Lapl\Psi=\Lapl\partial_t\Psi=\Lapl\tilde{\Psi}=\partial_t f$, i.e. $\partial_t \dslam  \Psi=G(\partial_t f)$. Similarly, $\partial_{x_i}\dslam  \Psi=G(\partial_{x_i} f)$, for $i=1,2$ and $(x_1,x_2)\in \Omega$.
\end{proof}

\subsection{Trace propositions for a family of weighted operators}
We obtain the following lemmas when considering the lateral boundary condition (A). We will use the extension on $\Omega$ to develop our trace propositions in weighted spaces on weighted operators. Furthermore, we will use trace propositions to specify the impact of the system's interior on the boundary. We first begin with a result for the density  of continuous functions in the spaces $\hh$ and $\hhh$.
\begin{lemma}\label{density}
The set $\mathcal{D}_0$ is dense in $\mathcal{H}$ and the set $\mathcal{D}_\lambda$ is dense in $\mathcal{H}^2$.
\end{lemma}
\begin{proof}
\begin{enumerate}
\item
Consider a function $v\in \hh$. Using the eigenfunctions defined in \cref{defeigenv}, we have that $v(z,x)=\sum\limits^\infty_{i=1}\alpha_i(z) e_i(x),$ while $\alpha_i(z)=\int_{\Omega} v e_i dx$. Therefore,
\begin{align*}
\norm{\dslam v}^2_{L^2} =\sum^{+\infty}_{k=1} k_i \int_{0}^\infty \alpha^2_i(z) dz+\sum^{+\infty}_{k=1}\int_0^{\infty} \lambda(z)\abs{\partial_z\alpha_i}^2 dz.
\end{align*}
Given a positive integer $N$, for each $z$ fixed, since $v$ in $L^2(\btp)$, we define a function $v_N$ in $L^2_z(\mathbb{R}_{+} C^\infty(\Omega))$,
\begin{align*}
v_N(z,x)=\sum_{i=1}^N\alpha_i(z) e_i(x).
\end{align*}
By definition, for each $z>0$,
\begin{align*}
&\quad\norm{v-v_N}^2_{L^2(\Omega)}+\norm{\dslam(v-v_N) }^2_{L^2(\Omega)}\\
&=\sum_{k=N+1}^{+\infty} \alpha_i^2+ \sum_{k=N+1}^{+\infty} k_i \alpha_i^2(z) +\sum_{k=N+1}^{+\infty}  \lambda(z)(\partial_z \alpha_i)^2 
\end{align*}
As $N\rightarrow \infty$, using Lebesgue dominate theorem,
\[
\norm{v-v_N}_{L^2(\btp)}+\norm{\dslam(v-v_N) }_{L^2(\btp)}\rightarrow 0.\]
In other words, $\norm{v-v_N}_{\hh}\rightarrow 0$. For each $\alpha_i$, $i\leqslant N$, for any $z,z_0>0$, as $z\rightarrow z_0$,
\begin{align*}
\abs{\alpha_i(z)-\alpha_i(z_0)} & \leqslant\abs{ \int_{z_0}^z \abs{\partial_z \alpha(\zeta)}d \zeta},\\
&=\abs{\int_{z_0}^z \abs{\partial_z \int_{\Omega}v(\zeta,x) e_i(x) dx }d \zeta},\\
&=\abs{\int_{z_0}^z \frac{1}{\sqrt{\lambda}}\abs{ \int_{\Omega}\sqrt{\lambda}\partial_z v(\zeta,x) e_i(x) dx }d \zeta},\\
&\leqslant \abs{\int_{z_0}^z \frac{1}{\sqrt{\lambda}}\norm{\sqrt{\lambda}\partial_z v(\zeta,x)}_{L^2(\Omega)} d \zeta},\\
&\leqslant \norm{\sqrt{\lambda}\partial_z v(\zeta,x)}_{L^2((z_0,z)\times\Omega)}\abs{\dfrac{z^{1-a}-z_0^{1-a}}{1-a}}^{1/2}\\
&\leqslant C\norm{\dslam v}_{L^2(\btp)}\abs{z^{1-a}-z_0^{1-a}}^{1/2}\rightarrow 0.
\end{align*}
Therefore, $\alpha_i$ is continuous for each $i$. Hence $v_N$ is $C^0_z(\mathbb{R}_+;(L^2(\Omega))$. Let $z_0=0$, we find that $\lim\limits_{z\rightarrow 0}\alpha_i(z)=\alpha_i(0)$, i.e, $v_N\in \dc$. Therefore, we show that $\dc$ is dense in $\mathcal{H}$.\\

\item
Similarly to 1., we first consider $u\in \hhh$, letting $u_N=\sum_{k=1}^N \alpha_i(z) e_i$, where $\alpha_i(z)=\int_{\Omega} u(z,x)e_i(x)dx$. For fixed $\varepsilon$, we claim that we can always find a $N$ big enough such that $\norm{\Lapl (u-u_N)}_{L^2}+\norm{\dslam(u-u_N)}_{L^2}<\varepsilon$. We have proved the convergence of the second part in 1. For fixed $z$,using Lebesgue dominated convergence theorem and $\norm{\Lapl u}^2_{L^2(\btp)}+\norm{\Lapl u_N}^2_{L^2(\btp)}<\infty$, we have that as $N\rightarrow \infty$:
\[
\norm{\Lapl(u-u_N)}^2_{L^2(\Omega)}=\norm{\sum_{i=N+1}^\infty\Lapl(\alpha_i(z) e_i(x))}_{L^2(\Omega)}\rightarrow 0.
\]
Further by Lebesgue dominant theorem, we have $\norm{\Lapl(u-u_N)}^2_{L^2(\btp)}\rightarrow 0$. Since $u_N=\sum_{k=1}^N \alpha_i(z) e_i$
\begin{align*}
\norm{\Lapl u_N}^2_{L^2(\btp)}&=\sum_{k=1}^N\int_{0}^\infty\left( \alpha_i(z)k_i+\zlz \alpha_i(z)\right)^2dz\\
&\geqslant \sum_{k=1}^N\int_{0}^\infty\left( -(\alpha_i(x)k_i)^2+\frac{1}{2}(\zlz \alpha_i)^2\right) dz	\end{align*}
Where $ \sum_{k=1}^N \int_{0}^\infty (\alpha_i(x)k_i)^2 dz =\norm{\overline{\Delta}u_N}_{L^2(\btp)}$ is finite. Therefore, the summation $\sum_{k=1}^N\int_{0}^\infty (\zlz \alpha_i)^2 dz	$ is finite as well. Hence, using embedding theorem, for each $i$, $\lambda\partial_z\alpha_i$ lies in the space $C^0_z(0,\infty)$. $u_N$ is therefore in the set $\ddd$ and $\ddd$ is dense in $\hhh$.
\end{enumerate}

\end{proof}
Using the smooth functions in set $\dc$ and $\ddd$, we can show that the Dirichlet extension operator minimizes the energy in the following lemma.
\begin{lemma}
\label{energy minimizer}
Given $f\in H^{\frac{1-a}{2-a}}(\Omega)$, let $u=E_2(f)$ the Dirichlet extension of $f$. Then $u$ minimizes the Dirichlet energy
\[
E(w)=\int_{\Omega\times \mathbb{R}_{+}}\abs{\dslam w}^2dxdz,\]
among all $w\in \hh$ with $\gamma_0 w=f(x)$ verifying the lateral boundary condition (A).
\end{lemma}
\begin{proof}
Let us first consider the periodic boundary condition. Let $\phi$ be a function in $\dc$ with $\gamma_0\phi=0$, since $\phi+u\in \hh$,
\[
\begin{array}{l}
E(u+\phi)-E(u)=\int_{\Omega\times\mathbb{R}_{+}}\dslam u\dslam \phi dxdz+\int_{\btp}\dslam \phi\dslam \phi dxdz,\\~\\

\qquad \geqslant -\int_{\Omega\times\mathbb{R}_{+}}\Lapl u\phi dxdz +\int_{\Omega\times\lbrace z=0\rbrace}\glmd u\phi dx.
\end{array}
\]

The last term vanishes since $\gamma_0\phi=0$.\\

For $\phi\in \dc$ and $\gamma_0\phi=0$  in the bounded case, 
\[
\begin{array}{l}
E(u+\phi)-E(u)=\int_{\Omega\times\mathbb{R}_{+}}\dslam u\dslam \phi dxdz+\int_{\btp}\dslam \phi\dslam \phi dxdz,\\~\\

\qquad \geqslant -\int_{\Omega\times\mathbb{R}_{+}}\Lapl u\phi dxdz +\int_{\partial\Omega\times \lbrace z>0\rbrace}\nu_s\cdot\gradb{} u\phi +\int_{\Omega\times\lbrace z=0\rbrace}\glmd u\phi dx.
\end{array}
\]
The last two terms vanish following the lateral boundary condition and $\gamma_0\phi=0$.
\end{proof}

\begin{proposition}
\label{trace}
\begin{enumerate}
\item
 $\gamma_0$ is a bounded operator mapping from the space $\hh$ to the space $\dot{H}^{\frac{1-a}{2-a}}(\Omega)$ with $\left\|\gamma_{0} u\right\|_{\dot{H}^{\frac{1-a }{ 2-a}}\left(\Omega\right)} \leqslant\|\nabla_{\slam} u\|_{L^{2}\left(\btp\right)}$
for $u $ such that $\|\nabla_{\slam} u\|_{L^{2}\left(\btp\right)}<\infty$,
where $\gamma_{0} u(x)=u(0, x)$ for $x \in \Omega$.

\item  $\gamma_{\lambda}$ is a bounded operator mapping from the space $\hhh$ to the space $\dot{H}^{-\frac{1-a}{2-a}}(\Omega)$ with  $\quad\left\|\gamma_{\lambda} u \right\|_{\dot{H}^{-\frac{1-a}{ 2-a}}\left(\Omega\right)} \leqslant\norm{\dslam u}_{[L^2(\btp)]^3}+\norm{\Lapl u}_{L^2(\btp)}$.

\item If $u\in \hh$ and $v\in \hhh$, then the following integration by parts is obtained:
\begin{equation}
\label{IBP}
\int_{\btp} \dslam u \cdot \dslam v dx dz =\int_{\btp} u\cdot(-\Lapl v)dx dz+ \int_{\btpz}\gamma_{\lambda} v \cdot\gamma_0 u dx,
\end{equation}
where $\gamma_{\lambda} v\in \dot{H}^{-\frac{1-a}{2-a}}$ and $\gamma_0 u\in  \dot{H}^{\frac{1-a}{2-a}}$ lies in the dual spaces. 
\end{enumerate}
\end{proposition}
\begin{proof}
\begin{enumerate}

\item
Let $u\in \dc$ and let $\tilde{u}:=E_2(\gamma_0(u))$. This $\tilde{u}$ minimizes the Dirichlet integral among functions with the same trace at $z=0$ due to the \cref{energy minimizer}. Hence
$$
\begin{array}{l}
\int_{0}^{\infty} \int_{\Omega }|\dslam u|^{2} d x d z \geq \int_{0}^{\infty} \int_{\Omega }|\dslam \tilde{u}|^{2} d x d z \\
\quad=\int_{\Omega } \tilde{u}(0, x) \gamma_{\nu}(\nabla_\lambda \tilde{u}) d x=\left\|\overline{\Delta}^{\frac{1-a}{2(2-a)}} \gamma_{0}(\tilde{u})\right\|_{L^{2}\left(\Omega \right)}^{2}=\left\|\gamma_{0}(u)\right\|_{\dot{H}^{\frac{1-a}{2-a}}\left(\Omega \right)}^{2} .
\end{array}
$$
The result can be extended to all $u \in \mathcal{H}\left(\mathbb{R}_{+} \times \Omega \right)$ by the density result in \cref{density}.
\item
Let $u,~v$ be functions such that $u\in \ddd$ and $v\in \dc$ with $\gamma_0 v=0$. Besides, let $v$ be a function compactly supported in $z$. Therefore,
\[
\int_{\Omega\times\lbrace z=0\rbrace }\gamma_0 v\glmd u=\int_{\btp }\dslam v \dslam u+\int_{\btp }v\Lapl u.\]
With the lateral boundary condition (A), using Poincar\'e 's inequality,
\begin{align*}
&\qquad\abs{\int_{\Omega }\gamma_0 v\glmd u}\\
&\leqslant (1+C_p) \norm{\dslam v}_{L^2(\btp)}\left(\norm{\dslam u}_{[L^2(\btp)]^3}+\norm{\Lapl u}_{L^2(\btp)}\right).
\end{align*}
Using 1. we have that
\[
\begin{array}{l}
\norm{\glmd u}_{H^{-\frac{1-a}{2-a}}}=\sup\left\lbrace\abs{\int_{\Omega }v\glmd u};\norm{v}_{H^{\frac{1-a}{2-a}}}\leqslant 1\right\rbrace,\\
\qquad \leqslant\sup\left\lbrace \abs{\int_{\Omega }v\glmd u};\norm{\dslam v}_{L^2(\btp )}\leqslant 1 \right\rbrace.
\end{array}\]

So $\norm{\glmd u}_{H^{-\frac{1-a}{2-a}}}\leqslant\norm{\dslam u}_{[L^2]^3}+\norm{\Lapl u}_{L^2}$.

\item
Let $\lbrace u_n\rbrace_{n=1}^\infty$ and $\lbrace v_n\rbrace_{n=1}^\infty$ be two sequences in $\dc$ and $\ddd$ and converge to $u,~v$ respectively. Then for fixed $n$, and any $z_0>0$, using divergence theorem, we have that on the bounded domain $[z_0,1/z_0]\times\Omega$, 
\begin{align*}
&\qquad \int_{\lbrace 1/z_0>z>z_0\rbrace \times \Omega}\dslam u_n\dslam v_n~ dz dx  \\
 & = \int_{\lbrace  1/z_0> z>z_0\rbrace \times \Omega} u_n\cdot(-\Lapl v_n) dz dx\\
& ~~~~ +\int_{\lbrace z=z_0 ,1/z_0\rbrace \times \Omega} \lambda(z)\partial_z v_n(z_0,x) \cdot u_n(z_0, x) dx,
\end{align*}
where $\lambda\in[z_0^a, z_0^{-a}]$ or $[z_0^{-a}, z_0^{a}]$. As $z_0\rightarrow 0$, we extend this equation to the whole space $\btp$ when
\[
\int_{\lbrace z=z_0 ,1/z_0\rbrace \times \Omega} \lambda(z)\partial_z v_n(z_0, x) \cdot u_n(z_0, x) dx\rightarrow 0+\int_{\Omega}\gamma_\lambda v_n\gamma_0 u_n dx,
\]
\[
\int_{\lbrace 1/z_0>z>z_0\rbrace \times \Omega}\dslam u_n\dslam v_n~ dz dx\rightarrow \int_{\btp}\dslam u_n \dslam v_n dz dx,
\]
and
\[
\int_{\lbrace  1/z_0>z>z_0\rbrace \times \Omega} u_n\cdot(-\Lapl v_n) dz dx \rightarrow  \int_{\btp}u_n\cdot(-\Lapl v_n)dzdx.
\]

And as $n\rightarrow \infty$, the integral on boundary converges to that of their limit $u$, $v$, since
\begin{align*}
&\abs{\int_{\Omega}\gamma_\lambda v_n\gamma_0 u_n dx-\int_{\Omega}\gamma_\lambda v\gamma_0 u dx}=\abs{\int_\Omega(\glmd v_n-\glmd v)\gz u_n +\glmd v(\gz u_n-u)},\\
&\leqslant\norm{\glmd v_n-\glmd v}_{\dot{H}^{-\frac{1-a}{2-a}}(\Omega)}\norm{\gamma_0 u_n}_{\dot{H}^{\frac{1-a}{2-a}}(\Omega)}\\
&\qquad +\norm{\glmd v}_{\dot{H}^{-\frac{1-a}{2-a}}(\Omega)}\norm{\gamma_0 u_n-\gamma_0 u}_{\dot{H}^{\frac{1-a}{2-a}}(\Omega)},\\
&\rightarrow 0.
\end{align*}

Similarly, we also have the other two terms converge as follows,
\[
\int_{\btp}\dslam u_n \dslam v_n dz dx\rightarrow \int_{\btp}\dslam u \dslam v dz dx,
\]
and
\[
\int_{\btp}u_n\cdot(-\Lapl v_n)dzdx \rightarrow \int_{\btp} u \cdot(-\Lapl v) dz dx.
\]
Therefore, we proved \eqref{IBP}.
\end{enumerate}
\end{proof}

\subsection{Transport equation}
Now let us specify some Lemmas on transport equations. These lemmas will be used to construct the potential vorticity in \eqref{3DQG} at time $t$.
\begin{lemma}
\label{Lpnorm}
Consider any weight function $w\in L^1_{loc}(\mathbb{R}_{+})$, given initial data $H_0\in L^p(\tbtp,w(z)dz)$ and a $2 D$ divergence-free velocity $u\in C^0(0,\infty; C^1(\tbtp))$, denoting $\tilde{\Omega}$ the whole domain $\ts$ or $\rt$, for $1\leqslant p\leqslant \infty$, there exist a unique solution $H\in C^1(0,\infty;L^p(\tbtp,w(z)dz))$ to the transport equation:
\[
\begin{cases}
\partial_{t} H+ u \cdot \overline{\nabla} H=0, &t>0, ~(z,x)\in\tbtp,\\
H(\cdot,0)=H_0, & t=0, ~(z,x)\in\tbtp.
\end{cases}
\]
Moreover, for all $t>0$, and $p\in(1,\infty)$, $$\int_{\tbtp} H^p(\cdot, t)w(z) dxdz=\int_{\tbtp} H_0^p w(z) dxdz,$$ and $\norm{H(\cdot,t)}_{L^\infty(\tbtp)}=\norm{H_0}_{L^\infty(\tbtp)}$.
\end{lemma}
\begin{remark}
When considering the bounded domain $\Omega$ in $\mathbb{R}^2$, we extend it to the whole space $\rt$. Thus we only consider \cref{Lpnorm} in $\ttp$ or $\rtp$.
\end{remark}
\begin{proof}
The existence and uniqueness of the solution are obtained using the classic transport theory. We prove the conservation of the weighted norms here. For any $1 \leqslant p< \infty$, let the transport equation times $p H^{p-1}$ and a weight function $w(z)$, and take integral in $\btp$,
$$
\begin{gathered}
\partial_{t} H^{p} w(z)+u \overline{\nabla} H^{p} w(z)=0, \\
\partial_{t} \int_{\tbtp} H^{p} w(z)+\int_{\tbtp} u \overline{\nabla} H^{p} w(z)=0.
\end{gathered}
$$
Thus
$$
\partial_{t} \int_{\tbtp} H^{p} w(z) d x d z=\int_{\mathbb{R}+} \int_{\tilde{\Omega}} \overline{\operatorname{div} u} \cdot H^{p} d x w(z) d z=0.
$$

Thus, $L^{p}$ norms conserve $\forall p \in[1, \infty)$. Hence $\forall t$,

\begin{align*}
\|H(\cdot, t)\|_{L^{\infty}(\tbtp)}=\lim _{p \rightarrow \infty}\|H(\cdot, t)\|_{L^{p}(\tbtp)}, \\=
 \lim _{p \rightarrow \infty}\|H(\cdot, 0)\|_{L^{p}(\tbtp)} = \|H(\cdot, 0)\|_{L^{\infty}(\tbtp)} .
\end{align*}

\end{proof}
~\\
\begin{lemma}
\label{Transport1}
Considering two transport systems on $\tbtp$ with a given constant $C$, 
\begin{equation}
\label{transport}
\begin{cases}
\partial_t f_i(t,z,x)+V_i(t,z,x)\cdot \nabla f_i(t,z,x)=0, & t>0,~(z,x)\in\tbtp\\
f_i(0,z,x)=f^{in}(z,x),& (z,x)\in(\tbtp),~i=1,2,
\end{cases}
\end{equation}
where $\nabla=(\partial_z,\partial_{x})$. The velocities $V_i(t,z,x)\in C^0(0,T;C^1(\tbtp))$ verifying $\nabla V_i\in C^0(0,T;L^\infty( \Omega\times\mathbb{R}_{+}))$.
Then when $T<1$, 
\[
\sup\limits_{[0,T]\times\tbtp}\abs{f_1-f_2}\leqslant CT\sup_{t\in[0,T]}\norm{\nabla V_1}_{L^\infty(\tbtp)}\norm{\nabla f^{in}}_{L^\infty(\tbtp)}\sup_{t\in[0,T]}\abs{V_1-V_2}.
\]
\end{lemma}

\begin{proof}
Denote that space variable $Y:=(z,x)$ and trajectory $X_i(s,t,z,x)\coloneqq\gamma_i(s)$, where $\gamma_i(s)$ satisfies the system of ODE,
\[
\begin{cases}
\dot{\gamma_i}(s)=V(s,\gamma_i(s)), &s>0\\
\gamma_i(t)=x, & \text{ for }i=1,2
\end{cases}
\]

Then considering the domain $\tbtp$,
\[
\begin{array}{l}
\abs{f_1-f_2}(t,Y)=\abs{f^{in}(X^1(0,t,Y))-f^{in}(X^2(0,t,Y))}\\~\\
\leqslant \norm{\nabla f^{in}}_{L^\infty(\tbtp)}\abs{X^1(0,t,Y)-X^2(0,t,Y)}\leqslant  \norm{\nabla f^{in}}_{L^\infty(\tbtp)}\abs{\gamma^1(0)-\gamma^2(0)}.
\end{array}
\]
And
\[\begin{array}{l}
\frac{d}{dt}(\gamma^1(t)-\gamma^2(t))=V_1(t,\gamma^1(t))-V_2(t,\gamma^2(t)),\\~\\
\qquad\qquad =V_1(t,\gamma^1(t))-V_1(t,\gamma^2(t))+V_1(t,\gamma^2(t))-V_2(t,\gamma^2(t)),\\~\\
\qquad\qquad\leqslant \sup_{t\in[0,T]}\norm{\nabla V_1}_{L^\infty(\btp)}\abs{\gamma^1(t)-\gamma^2(t)}+\sup_{t\in[0,T]}\abs{V_1-V_2}.
\end{array}
\]
Therefore using Gr\"onwall's inequality,
\[
\abs{\gamma^1(0)-\gamma^2(0)}\leqslant \sup_{t\in[0,T]} \abs{V_1-V_2}\left(\exp{\norm{\nabla V}_{L^\infty(\tbtp)}t}-1\right).\]
For $T<1$,
\[
\sup\limits_{[0,T]\times\tbtp}\abs{f_1-f_2}\leqslant T\norm{\nabla f^{in}}_{L^\infty(\btp,dz)} \sup_{t\in[0,T]}\norm{\nabla V_1}_{L^\infty(\btp)} \sup_{t\in[0,T]}\abs{V_1-V_2}.
\]

\end{proof}

\begin{lemma}
\label{Transport2}
For all $f^{in}\in W^{1,1}$ and $C>0$ such that $\norm{V}_{L^\infty_t(L^\infty(\tbtp))}+\norm{\nabla V}_{L^\infty_t(L^\infty(\tbtp))}\leqslant C$. There exists $\tilde{C}>0$ such that the following is true: The solution $f$ of the transport equation \eqref{transport} verifies $\norm{\nabla f}_{L^\infty_t(0,T;L^\infty(\tbtp))}\leqslant\tilde{C}_{V\vert_{t=0},f^{in}}\exp(CT)$.
\end{lemma}

\begin{proof}
As defined in the previous proof, $X(s;t,z,x):=\gamma(s)$ is the trajectory and $X$ solves the following transport equation:
\[
\partial_t X+V\cdot \gradb{}X=0.
\]
We have the following Gr\"onwall's inequality along the flow.
\[
\norm{\dfrac{D}{Dt}(\nabla X(0;t,z,x ) )}_{L^\infty(\tbtp)}\leqslant\norm{\nabla V}_{L^\infty(\tbtp)}\norm{\nabla X(0;t,z,x ) }_{L^\infty(\tbtp)}.
\]
Thus, for all $t>0$,
\[
\norm{\nabla X(0;t,z,x ) }_{L^\infty(\tbtp)}\leqslant \exp(Ct)\norm{\nabla X(0;z, x,0)}_{L^\infty(\tbtp)}.\]
Combined with
\[
\norm{\nabla f(t,z,x ) }_{L^\infty}=\norm{\nabla f^{in}(X(0;t,z,x ) )\cdot\nabla X(0;t,z,x ) }_{L^\infty}\]
\[\leqslant\norm{\nabla f^{in}(X(0;t,z,x ) )}_{L^\infty}\norm{\nabla X(0;t,z,x ) }_{L^\infty},
\]
we have that for all $t\in [0,T]$
\begin{align*}
&\qquad\norm{\nabla f(t,z,x ) }_{L^\infty(\tbtp)}\\
 & \leqslant \exp(Ct)\norm{\nabla X(0;z, x,0)}_{L^\infty(\tbtp)} \norm{\nabla f^{in}(X(0;z, x,T_0))}_{L^\infty(\tbtp)}\\
& \leqslant \tilde{C}\exp(CT)
\end{align*}
\end{proof}

\section{Fixed point}

In this section, we consider the modified system with $n$ fixed. We define a pair of  mollifiers $\eta_n\in C_c^\infty(\btp)$ and $\tilde{\eta}_n\in C_c^\infty([0,\infty)\times\btp)$. $\eta_n$ is positive compactly supported on $[0,2/n]\times[-{1/n},{1/n}]^2$ with $\int_{\btp}\eta_n\equiv 1$, and $\tilde{\eta}_n$ is positive compactly supported on $[0,2/n]\times[0,2/n]\times[-{1/n},{1/n}]^2$ with $\int_{\mathbb{R}_{+}\times\btp}\tilde{\eta}_n\equiv 1$. $\eta_n$, $\tilde{\eta}_n$ converges to the Dirac delta functions in $\btp$ and $[0,\infty)\times\btp$ respectively as $n\rightarrow \infty$. We fix the initial potential vorticity in the periodic cases as $F_0=\Lapl\Psi_0$. For the bounded case, we extend it as a function defined on $\rtp$ as:
\[
F_0=\begin{cases}
\Lapl\Psi_0 & x\in \Omega\\
0 & x\in \Omega^c
\end{cases}.
\]
Let $(F_0)_n= F_0 *\eta_n$, for both cases over $\tomega$. Assuming that the initial data $\Psi$ lies in the space $\hhh$, we naturally have that initially on the boundary $\theta_0\coloneqq\glmd \Psi_0\in H^{-\frac{1-a}{2-a}}(\Omega)$ and $\gz \Psi_0 \in H^{\frac{1-a}{2-a}}(\Omega)$.

For given $C^*>0$, we define two sets 
$\mathcal{C}_{int}(C^*)\coloneqq\lbrace F_0=\Lapl \Psi_{0}=\Lapl \Psi_{2,0} ~;\\ ~\norm{F_0}_{L^2(\tbtp)}+\norm{\nabla F_0}_{L^2(\tbtp)}\leqslant C^* \text{ and } \int_{\lbrace z\rbrace\times\ts}F_0= 0 \text{ for each }z>0 \text{ when } \Omega = \ts\rbrace$ and $\mathcal{C}_b(C^*)\coloneqq\lbrace \theta_0=\glmd \Psi_{0}=\glmd \Psi_{1,0} ; ~ \norm{\theta_0}_{L^2(\tilde{\Omega})}\leqslant C^*\rbrace$.
To show \cref{rFixedPoint}, for every $T<\infty$, we first prove the following theorem.
\begin{theorem}
\label{rFixedPoint}
(\textit{Fixed Point}) For any fixed $n$ and constant $C^*>0$, for every $T<\infty$, and initial data $F_0\in \mathcal{C}_{int}(C^*)$ and $\theta_0\in \mathcal{C}_b(C^*)$, there exists a solution $\Psi_n=\Psi_{1,n}+\Psi_{2,n}$ to the following equation with lateral boundary condition (A), verifying $\gradb{}\Psi_{1,n}$, $\gradb{}\Psi_{2,n} \in L^\infty_t(0, T; L^2(\Omega\times \mathbb{R}_{+}))$ respectively:
\begin{equation}\label{regularizedQG}
\begin{array}{l}
\qquad\qquad
\begin{cases}
\partial_t F_n+V_n \cdot\overline{\nabla}F_n=0, & \text{ in }\ttp \text{ or }\rt\times \lbrace z>0\rbrace,\\
\Lapl\Psi_{2,n}=F_n, &\text{ in }\Omega\times \lbrace z>0\rbrace,\\
\glmd \Psi_{2,n}=0, &  \text{ on } \Omega\times \lbrace z=0\rbrace,\\
F_n=(F_0)_n, & \text{ at } t=0.
\end{cases}\\
\text{And}\\
\qquad \qquad
\begin{cases}
\partial_t\theta_n+\mathbb{P}_n\left(\gradb{\perp}(\Psi_{2,n}+(-\overline{\Delta})^{\alpha-1}\theta_n)\cdot\gradb{}\theta_n\right)+\Lambda^{2\alpha}\theta_n =\mathbb{P}_n(\overline{\Delta}\Psi_{2,n}), \\
 \qquad \qquad \qquad \qquad \qquad \qquad \qquad \qquad \qquad \qquad\text{ in }\Omega\times \lbrace z=0\rbrace,\\
\Psi_1=E_1(\theta_n), \qquad \qquad\text{ in }\Omega\times \lbrace z>0\rbrace,\\
\theta_n=\PP n\theta_0,~~~\qquad \qquad\text{ at } t=0.
\end{cases}
\end{array}
  \end{equation}
where for the periodic case $V_n=\gradb{\perp}\Psi*\tilde{\eta}_n$ over $[0*T]\times\ttp$.  While for the bounded cases, let $V=\gradb{\perp}\Psi$ in $\Omega$ and let $V(x)=0$ on $\Omega^c$. That way, $V_n=V*\tilde{\eta}_n$ is defined as a function on $[0*T]\times\rtp$.
Here $\alpha=\dfrac{1}{2-a}$ and $\Lambda=\sqrt{-\overline{\Delta}_\Omega}$.
\end{theorem}

For convenience, we omit the index $n$ in this section.
\begin{theorem}
\label{contraction}
Given $C^*>0$, $T>0$ and initial data $F_0\in \mathcal{C}_{int}(C^*)$ and $\theta_0\in \mathcal{C}_b(C^*)$, and given
$\tilde{\Psi}_1$ and $\tilde{\Psi}_2$ in space $L^\infty_t(0,T;\hh)$, let $\tilde{V}=V'*\eta_n(z,x)$ , where $V'=\gradb{\perp}(\tilde{\Psi}_1+\tilde{\Psi}_2)$ for the periodic case, while for the bounded case:
\[
V'=\begin{cases}
\gradb{\perp}(\tilde{\Psi}_1+\tilde{\Psi}_2 ),& x\in \Omega,\\
0, & x\in \Omega^c.
\end{cases}
\]

Then there exist a time period $[0,t_0]$ such that the solution $\Psi_1$ and $\Psi_2$ exist in the space $\hhh$ solves the following equations:
\begin{equation}\label{contrationQG1}
\qquad\qquad
\begin{cases}
\partial_t F+\tilde{V} \cdot\overline{\nabla}F=0, & \text{ in }\ttp \text{ or }\rt\times \lbrace z>0\rbrace,\\
\Lapl\Psi_2=F, &\text{ in }\Omega\times \lbrace z>0\rbrace,\\
\glmd \Psi_2=0, &  \text{ on } \Omega\times \lbrace z=0\rbrace.
\end{cases}
\end{equation}
And 
\begin{equation}
\label{contrationQG2}
\begin{cases}
\partial_t\theta_n+\mathbb{P}\left((\gradb{\perp}\Psi_2+\gradb{\perp}(-\overline{\Delta})^{\alpha-1}\theta)\cdot\gradb{}\theta\right)+\Lambda^{2\alpha}\theta =\mathbb{P}(\overline{\Delta}\Psi_2), &  \text{ in }\Omega\times \lbrace z=0\rbrace,\\
\Lapl\Psi_1=0   \text{ in } \Omega\times \lbrace z>0\rbrace,\qquad \glmd \Psi(z,x)=\theta(x).
\end{cases}
\end{equation}

Hence there exists an operator $\Gamma$: $L^\infty(0,t_0;\hh(\btp))\rightarrow L^\infty(0,t_0;\hh(\btp))$ 
\[\Gamma(\tilde{\Psi}_1,\tilde{\Psi}_2)= (\Psi_1, \Psi_2),\] defined by solving the equations \eqref{contrationQG1} and \eqref{contrationQG2}. Moreover, $\Gamma$ is a contraction when $t_0<\delta_0(C^*)$.
\end{theorem}
We prove the following propositions to prove \cref{contraction}. When given $\tilde{\Psi}_1$ and $\tilde{\Psi}_2$, we first construct $F$ in \cref{Pinterior}, and for every time $t\in[0,t_0]$ we build a solution $\Psi_2$ to the degenerated elliptic equation with Neumann boundary condition. 
Then with the $\Psi_2$, we construct $\theta$ by solving the boundary equation in \cref{Pboundary}. Finally, we construct $\Psi_1$ using Neumann extension in \cref{PExtensionLemma} on a bounded convex domain in $\rt$ or Torus.

\subsection{Interior Equation}
We first prove \cref{Pinterior} to construct the interior part of the solution $\Psi_2$.
\begin{proposition}
\label{Pinterior}
Given divergence-free velocity $\tilde{V}\in C^1_t(0,T; C^\infty(\tbtp))$, a constant $C^*$ and the initial data $F_0\in \mathcal{C}_{int}(C^*)$, there exist a solution $\Psi_2\in L^\infty_t([0,T]\times\hhh)$ to the equation \eqref{contrationQG1}. Moreover, when given any divergence-free velocity $\tilde{V}_{1}$, $\tilde{V}_2 \in C^1_t(C^\infty(\tbtp))$ and $t_0$ small enough, there exist a constant $$\Theta_{1}=Ct_0\sup\limits_{t\in[0,t_0]}\norm{\nabla \tilde{V}}_{L^2(\tbtp)}\norm{\nabla F_0}_{L^2(\tbtp)},$$
such that the two different solutions verifying $$\norm{\Psi_{2,1}-\Psi_{2,2}}_{L^\infty([0,t_0],\hh)}\leqslant\Theta_{1}\norm{\tilde{V}_1-\tilde{V}_2}_{L^\infty([0,t_0],L^2(\tbtp)}.$$
\end{proposition}
\begin{proof}
Since the potential vorticity is transported by $\tilde{V}$, which is divergence-free and lies in the space $C^1_t(C^\infty(\tbtp))$, for any time $t>0$, there exist a unique solution $F\in C^1([0,t_0]\times\tbtp)$ given by the formula:
\[
F(t,z,x)=F_0(X(0,t,z,x)),
\]
for all $t\in [0,t]$ and all $(z,x)\in \tbtp$. 
Hence, at any time $t$, $\norm{F}_{L^2(\tbtp)}=\norm{F_0}_{L^2(\tbtp)}$. For each time $t\in [0,t_0]$, we solve the following singular elliptic equation with a Neumann boundary condition for $\Psi_2$:
\[
\begin{cases}
\Lapl\Psi_2= F, & \text{ on }\btp,\\
\glmd \Psi_2 =0.
\end{cases}
\]
Using \cref{Lweightedexistence}, there exists a unique solution $\Psi_2(t,\cdot)\in  \hhh$ for each time $t$. Therefore, a unique solution exists $ \Psi_2\in L^\infty_t([0,t_0]\times\hhh)$. Furthermore, $$\norm{\gradb{}\dslam \Psi_2}_{L^\infty_t([0,t_0]\times\hhh)}\leqslant\norm{F}_{L^\infty_t(L^2(\tbtp))}\leqslant\norm{F_0}_{L^2(\tbtp)}. $$

Besides, using \cref{Transport1}, when the length of time period $t_0$ small enough and given a pair of different velocity $\tilde{V}_1,~\tilde{V}_1\in C^1_t(C^\infty(\tbtp))$, we can find a constant $\Theta_1=Ct_0\sup\limits_{t\in[0,t_0]}\norm{\nabla \tilde{V}_1}_{L^2(\tbtp)}\norm{\nabla F_0}_{L^2(\tbtp)}$ such that
\[\begin{array}{l}
\norm{F_1-F_2}_{L^\infty(L^2(\tbtp))}
\leqslant \Theta_1 \sup _{t\in[0,t_0]}\norm{\tilde{V}_1-\tilde{V}_2}_{L^2(\tbtp)}.
\end{array}
\]
Hence, by \cref{PExtensionLemma},
\begin{equation}
\label{Interior}
\begin{array}{l}
\sup_{t\in[0,t_0]}\norm{\Psi_{2,1}-\Psi_{2,2}}_{\hh}
\leqslant \Theta_1\sup_{t\in[0,t_0]}\norm{\tilde{V}_1-\tilde{V}_2}_{L^2(\tbtp)}.
\end{array}
\end{equation}
\end{proof}

\subsection{Boundary Equation}
In this section, we use the solution $\Psi_2$ to \eqref{contrationQG1} obtained previously to show the existence of the solution on the boundary. Then constructing the other part of the solution $\Psi_1$ using extension Lemmas \ref{ExtensionLemma} and \cref{PExtensionLemma}.
\begin{proposition}
\label{Pboundary}
Consider the interior solution $\Psi_2\in \hhh$ and a constant $C^*$  and the initial data $\theta_0\in \mathcal{C}_b(C^*)$. There exist a unique $\Psi_1\in \hh$ solves the equation \eqref{contrationQG2}. Moreover, given any $\Psi_{2,1}$ $\Psi_{2,2}\in \hhh$, when $t_0<1$, there exist a constant $\Theta_2=C_{a,n,C^*}(e^{t_0}-1)$ such that,
\begin{equation}
\label{Boundarysmalltime}
\sup\limits_{\tau\in[0,t]}\norm{\theta_1-\theta_2}_{L^2({\Omega})}\leqslant \Theta_2 \sup\limits_{\tau\in[0,t]}\norm{\Psi_{2,1}-\Psi_{2,2}}_{\dot{H}^{1-\alpha}({\Omega})}.
\end{equation}
\end{proposition}
\begin{proof}
Let $\lbrace e_i\rbrace_{i=1}^\infty$ be the family of orthonormal basis defined by equations \eqref{defeigenv} and \eqref{defeigenc} with corresponding eigenvalues $\lbrace k_i\rbrace_{i=1}^\infty$ and $\PP n$ defined by \cref{projection}. Then equation \eqref{contrationQG2} can be rewritten as:

\begin{numcases}{} 
\partial_t\theta+\PP n\left((v+\gradb{\perp}\theta )\nabla \theta\right)+(-\overline{\Delta})^{\alpha}\theta=\PP n f, & $t>0$\label{boundaryeq1},\\
\theta=\PP n\theta_0, & $t=0$\label{boundaryeqi}.
\end{numcases}
where $v=\Psi_2$ and $f=\overline{\Delta}\Psi_2$. Using the Cauchy-Lipschitz theorem, there exists a unique solution $\theta\in C^1(0,t_0; C^\infty(\Omega))$ to this ordinary differential equation verifying the initial condition on the interval $[0,t_0]$.\\

Using the extension operator $E_1$ defined in \cref{PExtensionLemma}, there exists a unique $\Psi_1= E_1(\theta)\in\hh$ such that $\glmd \Psi_1=\theta$ and $\Lapl \Psi_1=0$ with lateral boundary condition (A). Moreover, $\norm{\Psi_1}_{\hh}\leqslant 2\norm{\dslam\Psi_1}_{L^2(\btp)}=\norm{\theta}_{\dot{H}^{-\frac{1-a}{2-a}}(\Omega)}.$ \\

Given different $v^1=\Psi_{2,1}$ and $v^2=\Psi_{2,2}$ and corresponding solution $\Psi_{1,1}$ and $\Psi_{1,2}$, we find that
\[\begin{array}{l}
\qquad\partial_t(\theta_1-\theta_2)\\
=-P_n\left(((v^1-v^2)+\gradb{\perp}(\theta_1-\theta_2 ))\nabla\theta_1 \right)-P_n\left((v^2+\gradb{\perp}\theta_2 )\nabla (\theta_1-\theta_2)\right)\\
\quad+(-\overline{\Delta})^{\alpha}(\theta_1-\theta_2)+P_n (f^1-f^2 ).
\end{array}
\]
If we multiply the equation with $\theta_1-\theta_2 $ and integrate on $\Omega$,
\begin{align*}
&\quad\partial_t\norm{\theta_1-\theta_2}^2_{L^2(\Omega)}\\
&\leqslant \int_{{\Omega}}\gradb{}(\Psi_{2,1}-\Psi_{2,2})\cdot\nabla\theta_1 (\theta_1-\theta_2) dx+\norm{\theta_1-\theta_2}^2_{\dot{H}^\alpha(\Omega)}+\int_{\ts}(f^1-f^2)(\theta_1-\theta_2)dx,\\
&\leqslant \norm{\gradb{}(\Psi_{2,1}-\Psi_{2,2})}_{\dot{H}^{-\alpha}(\Omega)}\norm{\theta_1}_{\dot{H}^{2+\alpha}(\Omega)}\norm{\theta_1-\theta_2}_{L^2(\Omega)}\\
&\quad+\norm{\theta_1-\theta_2}^2_{\dot{H}^\alpha(\Omega)}+\norm{f^1-f^2}_{\dot{H}^{-1-\alpha}(\Omega)}\norm{\theta_1-\theta_2}_{\dot{H}^{1+\alpha}(\Omega)},\\
&\leqslant \tilde{C}_{a,n,C^*}\left( \norm{\Psi_{2,1}-\Psi_{2,2}}^2_{\dot{H}^{1-\alpha}(\Omega)}+\norm{\theta_1-\theta_2}^2_{L^2(\Omega)}\right).
\end{align*}

Thus by Gronwall
\[
\norm{\theta_1-\theta_2}^2_{L^2}(t)\leqslant C_{a,n,C^*} e^t\int_0^t e^{-\tau} \norm{\Psi_{2,1}-\Psi_{2,2}}^2_{\dot{H}^{1-\alpha}(\ts)} d\tau.
\]

Since $t_0<1$ and let $\Theta_{2}=C_{a,n,C^*}( e^{t_0}-1)$ we proved the inequality \eqref{Boundarysmalltime}. 
\end{proof}
\subsection{Coupled equation}
In this section, we first prove \cref{contraction}. Then we make use of it to show the existence of the solution to the regularized system in \cref{rFixedPoint}.\\

Now let us prove \cref{contraction}.
\begin{proof}
Let $B_R$ be a ball in $L^\infty(0,t_0;\hh)$ with radius $R$. Given $R$, let us define $B_R:=\lbrace\Psi_1,\Psi_2\in [L^\infty(0,t_0;\hh)]^2;$ $ \norm{\gradb{}\Psi_1}_{L^\infty(0,t_0;L^2(\btp))}+\norm{\gradb{}\Psi_2}_{L^\infty(0,t_0;L^2(\btp))}\leqslant R\rbrace$. By \cref{Pinterior} and \cref{Pboundary}, for all $\Psi_1+\Psi_2\in B_R$, we can modify the velocity $\tilde{V}\in L^\infty(0,t_0;L^2(\tbtp))$ and there exist a pair of solution $(\Psi_1,\Psi_2)=\Gamma(\tilde{\Psi}_1,\tilde{\Psi}_2)$. So $\Gamma$ is well defined from $[L^\infty(0,t_0;\hh)]^2$ into $[L^\infty(0,t_0;\hh)]^2$. \\

For all $(\tilde{\Psi}^1_1,\tilde{\Psi}^1_2)$, $(\tilde{\Psi}^2_1,\tilde{\Psi}^2_2)~\in [L^\infty(0,t_0;\hh)]^2$, using Trace Propositions \ref{trace} and inequalities \eqref{Interior},
\begin{align*}
\sup\limits_{t\in[0,T]}\norm{\Psi_{2,1}-\Psi_{2,2}}_{\dot{H}^{1-\alpha}(\Omega)}&\leqslant\sup\limits_{t\in[0,T]}\norm{\dslam(\Psi_{2,1}-\Psi_{2,2})}_{L^2(\btp)},\\
&\leqslant \sup\limits_{t\in[0,T]}\norm{F^1-F^2}_{L^2(\btp)}.
\end{align*}
\[
\sup _{t \in[0, T]}\left\|\left(\Psi_{2,1}-\Psi_{2,2}\right)\right\|_{\mathcal{H}}\leqslant \Theta_1\sup _{t \in\left[0, t_0\right]}\left\|\left(\tilde{V}^1-\tilde{V}^2\right)\right\|_{L^2\left(\btp \right)}.
\]

Using \cref{equiv} and the extension \cref{PExtensionLemma}, given a constant $C_a$
\begin{align*}
&\qquad\sup\limits_{t\in[0,T]}\norm{\theta_1-\theta_2}_{L^2({\Omega})}\\
&=C_a \sup\limits_{t\in[0,T]}\norm{\theta_1-\theta_2}_{\dot{H}^{\frac{1-a}{2-a}}({\Omega})}=C_a\sup\limits_{t\in[0,T]}\norm{\dslam(\Psi_{1,1}-\Psi_{1,2})}_{L^2(\tbtp)}.
\end{align*}

Therefore by inequality \eqref{Boundarysmalltime},
\[
\sup\limits_{t\in[0,T]}\norm{\dslam(\Psi_{1,1}-\Psi_{1,2})}_{L^2(\tbtp)}\leqslant C_a\Theta_2\sup_{t\in[0,t_0]}\norm{\Psi_{2,1}-\Psi_{2,2}}_{\hh}. 
\]
Hence combining all the inequalities above,
\begin{equation}
\label{InequaContractino}
\sup\limits_{t\in[0,T]}\left(\norm{(\Psi_{2,1}-\Psi_{2,2})}_{\hh}+\norm{(\Psi_{2,1}-\Psi_{2,2})}_{\hh}\right)\leqslant\Theta\sup_{t\in[0,t_0]}\left(\norm{(\tilde{V}^1-\tilde{V}^2)}_{L^2(\btp)}\right),
\end{equation}
Where 
\begin{align*}
&\qquad\Theta=C_a(1+\Theta_2)\Theta_1\\
&=(1+C_{a,n,C^*}(e^{t_0}-1))Ct_0\sup_{t\in[0,t_0]}\norm{\nabla \tilde{V_1}}_{L^2(\tbtp)}\norm{\nabla F_0}_{L^2(\tbtp)}
<1,
\end{align*}
when $t_0$ is small.

Therefore the operator $\Gamma$ exists and is a contraction.
\end{proof}
Now let us finish proving \cref{rFixedPoint}.
\begin{proof}
If we multiply the equation \eqref{boundaryeq1} with $\theta$ and integrate over $\Omega$, we find that, for all $t\in [0,T]$,
\[
\frac{\partial}{\partial t}\norm{\theta}_{L^2}\leqslant \norm{\theta}^2_{\dot{H}^{\alpha}}+\norm{f}_{\dot{H}^{-\alpha-1}}\norm{\theta}_{\dot{H}^{\alpha+1}}\leqslant C_{n,\alpha} \norm{\theta}^2_{L^2}+\norm{f}^2_{\dot{H}^{-\alpha-1}}.\]
Thus by Gronwall, at time $t$, 
\begin{equation}
\norm{\theta}^2_{L^2}\leqslant e^t\int_0^{T}\left(\norm{f}^2_{\dot{H}^{-\alpha-1}}e^{-\tau}+\norm{\theta_0}_{L^2}e^\tau\right) d\tau.
\label{thetaL2}
\end{equation}
Since $\sup\limits_{t\in[0,T]}\norm{f}_{\dot{H}^{-\alpha-1}(\Omega)}=\sup\limits_{t\in[0,T]}\norm{\overline{\Delta} \Psi_2}_{\dot{H}^{\frac{1-a}{2-a}}(\Omega)}\leqslant\sup\limits_{t\in[0,T]}\norm{\gradb{}\dslam\Psi_2}_{L^2(\btp)}\leqslant\sup\limits_{t\in[0,T]}\norm{\Psi}_{\hhh}$ using the weighted Trace \cref{trace}, \cref{Lpnorm} and \cref{Lweightedexistence}. Therefore $\theta$ is in $L^\infty(0,T;L^2(\btp))$ .\\

For any $t\in [0,T]$, since the modified velocity $V$ is smooth and verifying:
\begin{align*}
\norm{V}_{L^\infty_t(0,T;L^\infty(\tbtp))}+\norm{\nabla V}_{L^\infty_t(0,T;L^\infty(\tbtp))}\leqslant\sup_{t\in[0,T]}( \norm{F}_{L^2(\tbtp)}+\norm{\theta}_{L^2(\Omega)}),
\end{align*}
using \cref{Lpnorm} and \cref{Transport2},
with $C=\norm{F_0}_{L^2(\tbtp)}+\norm{\theta}_{L^\infty(0,T;L^2(\Omega))})$, there exist a constant $\tilde{C}$ such that,
\[\sup_{t\in[0,T]}\norm{F}_{L^2(\tbtp)}+\norm{\nabla F}_{L^\infty_t(0,T;L^\infty(\tbtp))} \leqslant \norm{F_{0}}_{L^2(\tbtp)}+\tilde{C}exp(CT).
\]

Therefore, letting $k\in \mathbb{N}$ and $t_k=t_0\times k$, for each interval $[t_k, t_{k+1}]\in [0,T]$, there exist a constant $C^*_T$ such that at $t=t_k$, $F\in \mathcal{C}_{int}(C^*)$ and $\theta\in \mathcal{C}_b(C^*)$. Using \cref{rFixedPoint}, since $\Gamma$ is a contraction, there exist a unique solution $\Psi=\Psi_1+\Psi_2$ to \eqref{regularizedQG} in $L^\infty(0,t_0; \hh)$. Then by induction, the solution exists in $L^\infty(0, T; \hh)$.
\end{proof}

\section{Proof of Theorem 1.1}
This section is dedicated to the proof of \cref{main}. 
\begin{proof}
We split the proof into several steps.
\begin{steps}
\item
Define a continuous weight function $w_0$ mapping $\mathbb{R}$ to $\mathbb{R}$ such that $w_0(z)$ strictly decrease and converges to $0$ as $z\rightarrow \infty$. And $w_0(z)=1$ when $z\leqslant 1$. For given $C>0$, we define the following three spaces for applying Aubin Lions Lemma:\\

\begin{itemize}

\item $X_{-1}=\left\lbrace \gradb{\perp}\Psi_2 = \overline{\divg}M(z,x_1,x_2)~\vert~ \norm{M}_{[L^2(\btp)]^4}\leqslant C_{-1},\glmd \Psi_2=0 \right\rbrace$,\\ equipped with the norm:

\medskip

$\norm{\gradb{\perp}\Psi_2}_{X_{-1}}=\inf\limits_{\substack{\norm{M}_{[L^2(\btp)]^4}\leqslant C_{-1},\\ \gradb{\perp}\Psi_2=\overline{\divg} M}}\left(\iint_{\btp} \abs{M(z,x)}^2 w_0(z)dx dz\right)^{1/2}$.

\item $X_0=\left\lbrace\gradb{\perp}\Psi_2:\norm{\gradb{}\Psi_2}_{L^2(\btp)}\leqslant C_0,~\glmd \Psi_2=0\right\rbrace$, equipped with the norm:

\medskip

$\norm{\gradb{\perp}\Psi_2}_{X_0}=\left(\iint_{\btp}\abs{\gradb{\perp}\Psi_2}^2 w_0(z)~dx dz\right)^{1/2}$.

\item $X_{1}=\left\lbrace \gradb{\perp}\Psi_2 : \norm{\Lapl\Psi_2}_{L^2(\btp)}\leqslant C_{1},~\glmd \Psi_2=0\right\rbrace$, equipped with the norm:
\medskip

$\norm{\gradb{\perp}\Psi_2}_{X_1}=\norm{\Lapl\Psi}_{L^2(\btp)}.$
\end{itemize}
Using \cref{Lweightedexistence} and $\glmd \Psi_2=0$,  $$\norm{\gradb{\perp}\Psi_2}_{L^2(\btp, w_0dz)} \leqslant\norm{\Lapl\Psi_2}_{L^2(\btp)}.$$
Letting $C_0=C_1$, we have $\norm{\gradb{\perp}\Psi_2}_{X_0}\leqslant\norm{\Lapl\Psi_2}_{X_{-1}}$. Moreover, with boundary condition (A), by Poincar\'e inequality for each $z>0$, $\norm{\Psi_2}_{L^2(\lbrace z \rbrace\times \Omega)}\leqslant C_p\norm{\gradb{}\Psi_2}_{L^2(\lbrace z \rbrace\times \Omega)}$, where $C_p$ only depends on the domain $\Omega$ as defined previously. Note that for $M=\left(\begin{matrix}
0 & -\Psi_2\\ \Psi_2 & 0
\end{matrix}\right)$, we have that $\gradb{\perp}\Psi_2=\overline{\divg}M$. Therefore, $\norm{\gradb{}\Psi_2}_{X_{-1}}\leqslant\norm{\gradb{}\Psi_2}_{X_{0}}$ if letting $C_0=C_p\times C_{-1}$. Hence $X_{1}\subset X_0\subset X_{-1}$.\\

Then we first show the space $X_1$ is compactly embedded in the space $ X_0$. Considering $g_m=\gradb{\perp}\Psi_{2,m}$ a bounded sequence in $X_1$ and letting $C_{1}=\norm{\Lapl\Psi_{2,0}}_{L^2(\btp)}$, we have $$\norm{g_m}_{L^2(\btp)}+\norm{\dslam g_m}_{L^2(\btp)}\leqslant \norm{\Lapl \Psi_{2,m}}_{L^2(\btp)}\leqslant C_1,$$ uniformly for all $m$. Thus there exist a $g=\gradb{\perp}\Psi_{2}$ such that up to a subsequence, $g_{mp}$ converge weakly to $g$ in $L^2(\btp)$.Moreover, for all $m$,
\begin{align*}
\norm{g_m(z,\cdot)-g_m(0,\cdot)}^2_{L^2(\Omega)}=\int_\Omega\left(\abs{\int^z_0\dfrac{1}{\sqrt{\lambda}}\sqrt{\lambda}\partial_z(g_{m}(\tau,\cdot)-g_{m}(0,\cdot)  )d\tau}^2\right)dx,\\
\leqslant \norm{\dslam g_m}_{L^2(\btp)} \sqrt{\int_0^z\frac{d\tau}{\lambda(\tau)}}\leqslant C_1 z^{1-a},\qquad a<1.
\end{align*}
As $z\rightarrow 0$, $z^{1-a}$ also approaches $0$. For bounded convex domain $Q_j\coloneqq\Omega\times [1/j, j]$, $j\in \mathbb{N}^*$, there exists a constant $C_j$ such that for all $m$, 
\[
\norm{\dslam(g-g_m)}_{L^2(Q_j)}\geqslant C_j \norm{\nabla (g-g_m)}_{L^2(Q_j)}.\]
Thus the weak convergent sequence $\lbrace g_{jm}\rbrace_{m=1}^\infty $ also converges strongly to $g$ in $L^2$ over $Q_j$ by embedding theorem, by the uniqueness of limit. Furthermore, for given $\varepsilon>0$ and $j$ sufficient big such that $\frac{1}{\sqrt{2-a}} C_1 (1/j)^{1-a/2}<\frac{\varepsilon}{3}$ and  $w(j)<\frac{\varepsilon}{6C_1}$, there exists $P>0$,such that $\forall p\geqslant P$, $\norm{g-g_{mp}}_{L^2(\mathbb{R}_{+}\times Q_j)}\leqslant \varepsilon/3$. Then, we have  
\begin{align*}
&\qquad\norm{g_{mp}-g}_{L^2(Q_j^c,w_0 dxdz)}\\
&= \norm{g_{mp}-g}_{L^2_z(L^2(z<1/j,dxdz))}+\norm{g_{mp}-g}_{L^2_z ((L^2(z>j,w_0dxdz)))},\\
&\leqslant \frac{1}{\sqrt{2-a}} C_1 (1/j)^{1-a/2}+2C_1 w_0(j)\leqslant 2\varepsilon/3.
\end{align*}
Then $\norm{g_{mp} -g}_{L^2(\btp)w_0dz}<\varepsilon$. Hence there exist a subsequence $\lbrace g_{mp}\rbrace_{j=1}^\infty$ of $\lbrace\gradb{\perp}\Psi_{2,m}\rbrace_{m=1}^\infty$ converges strongly to $\gradb{\perp}\Psi_2$ in $X_0$, i.e. $X_1$ is embedded in $X_0$ compactly.\\

We show that $X_0$ is continuously embedded into $X_{-1}$. If there is a convergence sequence $\gradb{\perp}\Psi_{2,m}\rightarrow \gradb{\perp}\Psi_2$ in $X_0$, we need to show that $\gradb{\perp}\Psi_{2,m}\rightarrow \gradb{\perp}\Psi_2$ in $X_{-1}$. For every $z>0$ fixed, $$\norm{\Psi_2-\Psi_{2,m}}^2_{L^2(\lbrace z\rbrace\times\Omega)}\leqslant \norm{\gradb{}\Psi_2-\gradb{}\Psi_{2,m}}^2_{L^2(\lbrace z\rbrace\times\Omega)},$$ by Poincar\'e inequality. Then $\iint_{\btp} \abs{\Psi_2-\Psi_{2,m}}^2 w_0dxdz $\\ $\leqslant C_p\iint_{\btp}\abs{\gradb{\perp}\Psi_2-\gradb{\perp}\Psi_{2,m}}^2  w_0dzdx$.
Thus letting $M=\left(\begin{matrix}
0 & -\Psi_2\\ \Psi_2 & 0
\end{matrix}\right)$, we have $\gradb{\perp}\Psi_2=\overline{\divg}M$. Therefore, $X_0$ is embedded in $X_{-1}$ continuously.\\

\item 
Using \cref{Lpnorm}, and initial potential vorticity $\Lapl \Psi_0$ lies in $L^\infty(\btp)$, for $t>0$, for the modified potential vorticity in \eqref{regularizedQG},
\begin{gather}
\norm{F_n}_{L^\infty(\tbtp)}\leqslant\norm{F_{n,0}}_{L^\infty(\tbtp)}\leqslant\norm{\eta_{\varepsilon}}_{L^1(\tbtp)}\norm{F_{n,0}}_{L^\infty(\tbtp)},\label{eqn:uniformvorticity}\\
\leqslant\norm{\Lapl\Psi_{2,0}}_{L^\infty(\btp)}=\norm{\Lapl\Psi_0}_{L^\infty(\btp)},\notag
\end{gather}
and 
\begin{multline}
\label{eqn:uniformvelocity}
\norm{V_n}_{L^\infty_t(0,T;L^2(\tbtp)}\leqslant \norm{\tilde{\eta}_\varepsilon}_{L^1((\mathbb{R}_{+}\times\btp))}\times\\
 \left(\norm{\dslam \Psi_{1,n}}_{L^\infty_t(L^2(\btp))}+\norm{\dslam \Psi_{2,n}}_{L^\infty_t(L^2(\btp))}\right).
\end{multline}

Applying the operator $G$ defined in \cref{Lweightedexistence} for the first equation of \eqref{regularizedQG} over $\btp$:
\[
\left(\partial_t+ \gradb{\perp}V_n\cdot\overline{\nabla}\right)F_n=0,
\]
it becomes,
\[
\partial_t G(F_n)+\overline{\divg}\left(G(V_n\cdot F_n)\right)=0.
\]
Hence, $\partial_t\gradb{\perp}\Psi_{2,n}\in L_t^\infty(0,T;X_{-1})$ and $\norm{\gradb{\perp}\Psi_{2,n}}){X_{-1}}\leqslant C_{-1}=C_1/C_p$. Moreover, 
\begin{align*}
&\quad\norm{\Lapl\Psi_2}_{ L_t^\infty(0,T;L^2(\btp))}\leqslant\norm{F_n}_{ L_t^\infty(0,T;L^2(\tbtp))},\\
&\leqslant \norm{F_{n,0}}_{L^2(\tbtp)}\leqslant \norm{\Lapl\Psi_{2,0}}_{L^2(\btp)},\\
&\leqslant \norm{\Lapl\Psi_{0}}_{L^2(\btp)},
\end{align*}
uniformly by \cref{Lpnorm}. Using Aubin Lion's Lemma, there exists a subsequence of $\gradb{\perp}\Psi_{2,n}$ converges strongly in $C^0_t(X_0)$. Using \cref{trace},  $\gamma_0\gradb{\perp}\Psi_{2,n} \xrightharpoonup{} \gamma_0\gradb{\perp}\Psi_{2}$ in $L^\infty_t(0,T;H^{\frac{1-a}{2-a}}(\Omega))$ in weak star sense. Hence, we summarize convergence terms as follows,
\begin{align}
\gradb{\perp}\Psi_{2,n}\rightarrow \gradb{\perp} \Psi_2, \qquad & \text{ in }C^0_t(X_0)=C^0_t(0,T;L^2(\btp,w_0dz)).\label{conpsi2}\\
\Lapl\Psi_{2,n} \xrightharpoonup{weak*} \Lapl\Psi_{2}, \qquad& \text{ in }L^\infty_t(0,T;L^2(\btp))\text{ in weak* sense}.\label{conlapl}
\end{align}
\begin{align}
\gamma_0\gradb{\perp}\Psi_{2,n} \xrightharpoonup{weak*} \gamma_0\gradb{\perp}\Psi_{2}, \qquad& \text{ in }L^\infty_t(0,T; L^2(\Omega))\text{ in weak* sense}.\label{convb2}
\end{align}

\item
Let us derive the energy inequality on the boundary.  Using trace \cref{trace}, we found that for all $n$ and $\gradb{\perp}\Psi_2\in X_0\cap X_1$,  $v_n = \gamma_0 \gradb{\perp}\Psi_{2,n}\in L^\infty(0,T;\dot{H}^{1-\alpha}(\Omega))=\dot{H}^{\frac{1-a}{2-a}}(\Omega)$ and $f_n=\overline{\Lapl}\Psi_2 \in L^\infty(0,T;\dot{H}^{-\alpha}(\Omega))$ and $\norm{f}_{\dot{H}^{-\alpha}(\Omega)}+\norm{\gamma_0\gradb{\perp}\Psi_{2,n}}_{ \dot{H}^{1-\alpha}(\Omega)}$ is bounded uniformly for all $t\in [0,T]$. Thus we have that
\[\begin{cases}
\frac{1}{2}\partial_t\norm{\theta_n}^2_{L^2(\Omega)}+\norm{\theta_n}^2_{\dot{H}^\alpha(\Omega)}\leqslant \frac{1}{2}\norm{f_n}^2_{\dot{H}^{-\alpha}(\Omega)}+\frac{1}{2}\norm{\theta_n}^2_{\dot{H}^\alpha(\Omega)},& 0<t<T.\\
 \norm{\PP n\theta_{0}}^2_{L^2(\Omega)}\leqslant \norm{\theta_{0}}^2_{L^2(\Omega)}, & t=0.
\end{cases}
\]
Therefore, for a.e. $t\in [0,T]$,
\[
\norm{\theta(t)}^2_{L^2(\Omega)}+\int_0^t \norm{\theta(\tau}^2_{\dot{H}^\alpha(\Omega)}d\tau \leqslant \norm{\theta_0}^2_{L^2(\Omega)}+\int_0^t \norm{f}^2_{\dot{H}^{-\alpha}(\Omega)}d\tau.\]

Thus, $\PP n \theta$ is uniformly bounded in $L^\infty_t(0,T;L^2(\btp))\cap L^2_t(0,T;\dot{H}^\alpha( \btp))$ for all $n$.
For $\varphi\in H^{2+{\varepsilon}}([0,\infty)\times \Omega)$, we have for $t$ fixed,

\begin{align*}
&\qquad\abs{\int_{\Omega} P_{n}\left(\left(v_n+\overline{\nabla}^{\perp} (-\overline{\Delta})^{\alpha-1}\theta_{n}\right) \nabla \theta_{n}\right)\varphi dx}\\
&=\abs{\int_{\Omega} \nabla(P_{n}\varphi)\left(\left(v_n+\overline{\nabla}^{\perp}(-\overline{\Delta})^{\alpha-1} \theta_{n}\right)  \theta_{n}\right) dx},\\
& \leqslant\norm{\nabla(P_{n}\varphi)}_{L^\infty}\norm{\theta_n}_{L^2(\Omega)}\left(\norm{v_n}_{L^2(\Omega)}+\norm{\theta_n}_{\dot{H}^{2\alpha-1}}\right), \\
& \leqslant\norm{\nabla(P_{n}\varphi)}_{L^\infty}\norm{\theta_n}_{L^2(\Omega)}\left(\norm{v_n}_{\dot{H}^{1-\alpha}}+\norm{\theta_n}_{\dot{H}^{\alpha}}\right),
\end{align*}

using Poincar\'e inequality and $2\alpha-1<\alpha$. At the same time,
\[
\norm{\Lambda^{2\alpha}\theta_n}_{\dot{H}^{-(2+{\varepsilon})}}=\norm{\theta_n}_{\dot{H}^{2\alpha-(2+{\varepsilon})}}\leqslant \norm{\theta_n}_{L^2}\in L^\infty[0,T],\]
\[
\norm{P_n f}_{\dot{H}^{-(2+{\varepsilon})}}\leqslant \norm{f}_{\dot{H}^{-\alpha}}\in L^\infty([0,T]).
\]

Therefore using Sobolev embedding theorem,  $\dot{ \theta}\in L^2_t(\dot{H}^{-(2+\varepsilon)})$.  \\

Let $X_0= H^{\alpha-1}(\Omega)$, $X_{1}=L^2(\Omega)$ and $X_{-1}=H^{-(2+\varepsilon)}$ with lateral boundary $(A)$. It is obvious that $X_1$ is compactly embedded in $X_0$ and $X_0$ is continuously embedded in $X_{-1}$.  We have shown that $\theta\in L_t^\infty(0,T;X_1)$ and $\partial_t \theta \in L_t^2(0,T; X_{-1})$ uniformly for all $n$. Therefore, using generalized Aubin-Lions Lemma, $\theta_n$ converges strongly in $L_t^2(0,T;X_0)=L_t^2 (0,T; H^{\alpha-1})$.

In summary of the boundary variables, we have already shown that
\begin{align}
\theta_n\rightarrow \theta, \qquad & \text{ in }L^2_t(0,T,\dot{H}^{\alpha-1}(\Omega))=L^2_t(\dot{H}^{-\frac{1-a}{2-a}}(\Omega)).\label{contheta}\\
\theta_n\rightharpoonup \theta, \qquad & \text{ weakly in } L^2_t(0,T; H^\alpha (\Omega)),\label{conthetaw}\\&\text{ and in }L^\infty_t(0,T;L^2(\Omega))\text{ in weak * sense}.\notag\\
\PP n\varphi\rightarrow \varphi, \qquad& \text{ in }C^0_t(0,T;\dot{H}^{2+{\varepsilon}}(\Omega)).\label{conphi}
\end{align}
Using \cref{PExtensionLemma} and \eqref{contheta}, we have that the Neumann extension $\Psi_{1,n}$ of $\theta_n$ converges strongly to $\Psi_1$ in $L^2_t (0,T;\mathcal{H})$. Hence,
\[
\gradb{\perp}\Psi_{1,n}\rightarrow \gradb{\perp}\Psi_{1}, \qquad \text{ in }L^2(0,T;L^2(\btp)).
\]
Accordingly with \eqref{conpsi2},
\begin{equation}
V_n\rightarrow \gradb{\perp}(\Psi_{1}+\Psi_{2}), \qquad \text{ in }L^2(L^2(\btp)).\label{conpsi}
\end{equation}
\item
Thanks to \eqref{conlapl} and \eqref{conpsi}, the nonlinear terms of the weak form of the bulk against test function $\phi\in C^\infty_c(\btp)$ converges to the limit as:
\begin{align*}
&\qquad -\int_0^T\int_{\btp}\left(V_n\cdot\overline{\nabla}\phi\right) F_n+\int_0^T\int_{\btp}(\overline{\nabla}^{\perp}\Psi_1+\overline{\nabla}^{\perp}\Psi_2)\cdot\overline{\nabla}\phi\Lapl\Psi_2,\\
& =\int_0^T\int_{\btp}(V_n-\overline{\nabla}^{\perp}(\Psi_1+\Psi_2))\cdot\gradb{}\phi F_n\\
&\quad+(\overline{\nabla}^{\perp}(\Psi_1+\Psi_2))\cdot\gradb{}\phi \left(F_n-\Lapl\Psi_2\right),\\
&\to 0.
\end{align*}

Consider now a new test function $\varphi\in C^\infty(\Omega)$ on the boundary $z=0$. We treat the nonlinear term on the boundary $z=0$ in the following way:
\[
\begin{array}{l}
\qquad\int_0^T\int_{\Omega } \left( \gradb{\perp}(-\overline{\Delta})^{\alpha-1}\theta+\gradb{\perp}\Psi_{2} \right)\theta \nabla \varphi d x\\
\qquad-\int_{\Omega } \left( \gradb{\perp}(-\overline{\Delta})^{\alpha-1}\theta_n+\gradb{\perp}\Psi_{2,n} \right) \theta_{n} \nabla \PP{n} \varphi d x,\\~\\
=\int_0^T\int_{\Omega } \gradb{\perp}(-\overline{\Delta})^{\alpha-1}\left( \theta-\theta_n\right)\theta \nabla \varphi+ (\gradb{\perp}\Psi_{2}-\gradb{\perp}\Psi_{2,n})\theta\nabla\varphi~ dxdt\\~\\
\qquad+\int_0^T\int_{\Omega } \gradb{\perp}(-\overline{\Delta})^{\alpha-1}\theta_n(\theta-\theta_n) \nabla \varphi + \gradb{\perp}\Psi_{2,n}(\theta-\theta_n) \nabla \varphi ~ dxdt\\~\\
\qquad +\int_0^T\int_{\Omega }\gradb{\perp}\left((-\overline{\Delta})^{\alpha-1}\theta_n+\gradb{\perp}\Psi_{2,n}\right) \cdot \theta_n\nabla (\varphi-\PP{n} \varphi) ~ d x dt,\\~\\
:= A_1+A_2+A_3.\\~\\
\end{array}
\]

From \eqref{conthetaw} and \eqref{convb2}, using that $\theta\nabla \varphi\in  L^2(0,T;H^{\alpha-1}(\Omega))$ by \eqref{contheta}, $A_1$ converges to $0$. Thanks to \eqref{contheta}, \eqref{conthetaw} and \eqref{convb2} $A_2$ also converges to $0$. Similarly, thanks to the strong convergence in $L^2(0,T; H^{\alpha-1}(\Omega))$ thanks to \eqref{convb2}, weak convergence thanks to \eqref{conthetaw}, $\gradb{\perp}\Psi_{2,n}$, $\gradb{\perp}(-\overline{\Delta})^{\alpha-1}\theta_n$ and $\theta_n$ are uniformly bounded in $L^2(0,T; L^2(\Omega))$. And thanks to \eqref{conphi} and Sobolev inequality, $\nabla (\varphi-\PP n \varphi)$ converges to $0$ in $C^b([0,T]\times\Omega)$ with $\gradb{\perp}\left((-\overline{\Delta})^{\alpha-1}\theta_n+\gradb{\perp}\Psi_{2,n}\right) \cdot \theta_n$ bounded in $L^1(0,T;L^1(\Omega))$. Hence $A_3$ converges to $0$ too. 
\end{steps}

\end{proof}

\newpage
\bibliographystyle{plain}
\bibliography{reference}

\end{document}